\documentclass[11pt]{article}

\usepackage{amsmath,amsfonts,amsthm,amssymb} 
\usepackage{tikz}
\usepackage{mathtools}
\usetikzlibrary{positioning}
\usetikzlibrary{snakes}
\usepackage{tikz,times}
\usepackage{geometry}
\usetikzlibrary{arrows,intersections}
\usetikzlibrary{mindmap,backgrounds}
\usepackage{comment}

\newtheorem{theorem}{Theorem}[section]
\newtheorem{proposition}[theorem]{Proposition}
\newtheorem{corollary}[theorem]{Corollary}
\newtheorem{lemma}[theorem]{Lemma}

\theoremstyle{definition}

\theoremstyle{remark}

\def\be{\begin{eqnarray}}
\def\ee{\end{eqnarray}}
\def\ben{\begin{eqnarray*}}
\def\een{\end{eqnarray*}}

\numberwithin{equation}{section}

\begin{document}

\title{On the Steady State of Continuous Time\\ Stochastic Opinion Dynamics with Power Law Confidence}

\author{
Jae Oh Woo\thanks{Department of Mathematics, and Department of Electrical and Computer Engineering, The University of Texas at Austin. Email: jaeoh.woo@aya.yale.edu}, 
Fran\c{c}ois Baccelli\thanks{Department of Mathematics, and Department of Electrical and Computer Engineering, The University of Texas at Austin. Email: baccelli@math.utexas.edu} 
and Sriram Vishwanth\thanks{Department of Electrical and Computer Engineering, The University of Texas at Austin. Email: sriram@ece.utexas.edu}}




\date{\today}

\maketitle
\begin{abstract}
This paper introduces a class of non-linear and continuous-time opinion dynamics model
with additive noise and state dependent interaction rates between agents.
The model features interaction rates which are proportional to a negative power of opinion distances.
We establish a non-local partial differential equation for the distribution of opinion distances
and use Mellin transforms to provide an explicit formula for the stationary solution of the latter,
when it exists. Our approach leads to new qualitative and quantitative
results on this type of dynamics. To the best of our knowledge these Mellin transform results
are the first quantitative results on the equilibria of opinion dynamics with distance-dependent
interaction rates. The closed form expressions for this class of dynamics are obtained for the two agent case.
However the results can be used in mean-field models featuring several agents
whose interaction rates depend on the empirical average of their opinions. 
The technique also applies to linear dynamics, namely with a constant interaction rate,
on an interaction graph.
\end{abstract}



\section{Introduction}
It is now well recognized that the polarization of opinions in societies is linked to the
fact that opinion updates are predominantly due to interactions between
persons having already closeby opinions. Several models of opinion dynamics 
incorporate this fact, for instance the bounded confidence model, also known as
the Hegselmann-Krause model \cite{hegselmann2002opinion},
where interactions take place between nodes with opinions less than a given confidence radius.
However none of these models offers an analytical framework allowing
one to quantitatively predict how opinions cluster under such dynamics.

The present paper proposes a new model incorporating the closeby opinion interaction nature by proposing a power law confidence model. This model is stochastic in two ways: 
first, interactions between any pair of agents
take place at epochs of a random point process with a stochastic intensity 
that decreases with the opinion discrepancy between the two agents; secondly, the opinion
of each agent evolves as a diffusion process. The main result of the present paper is 
that when the interaction rate decreases like a power law of the opinion discrepancy,
the model is analytically tractable. This tractability goes with several qualitative
results like, e.g., the conditions for stability, those for the lack of accumulation 
of interactions, or the effect of parameter changes on the stationary distribution when it exists. 

The analytical framework considered in this paper covers both the setting
where interactions between agents happen as an autonomous point process and that
where the rates of interactions depend on the discrepancies between opinions.
In both settings, we represent opinions by the points of the real line,
and we describe the geometry of opinions by the pairwise differences between opinions.
In the first setting, interactions take place independently of opinion values,
whereas in the second, they are governed by the current geometry of opinions in 
addition to governing the evolution of this geometry.
Models in the first class are in fact linear and there is a rich analytical
literature about scaling limits, convergence rate, stationary solutions on the matter
- see, e.g., \cite{toscani2006kinetic, fagnani2008randomized, olshevsky2009convergence, acemoglu2011opinion, acemouglu2013opinion, ghaderi2013opinion, yildiz2013binary}.
As mentioned above, the second class is much more difficult to analyze and the most noticeable results
concern the convergence of the dynamics to fixed points and
estimates on the speed of this convergence \cite{blondel2010continuous, lobel2011distributed,
como2011scaling, mirtabatabaei2012opinion, acemoglu2014state, brugna2015kinetic, baccelli2015pairwise}.

Most papers on the matter bear on the deterministic and discrete-time case.
The present paper is focused on a stochastic, continuous-time model.
The stochasticity first comes from the presence of an additive noise featuring
self-beliefs and represented as an i.i.d. sequence in \cite{baccelli2015pairwise}.
The second source of stochasticity lies the fact that interactions happen at the epochs of a random point
process with a stochastic intensity function of the opinion distances.
In the power law confidence model proposed here, this interaction rate is proportional
to a negative power of the opinion distance. For this parametric setting, one gets a
representation of the dynamics of opinion distances in terms of a stochastic differential
equation, and a non-local partial differential equation for the distribution of this
stochastic process. The main analytical novelty lies in the use of Mellin transforms
to solve the stationary equation in a closed-form. This continuous time parametric model is used
for tractability reasons in the first place; it allows us to use the powerful analytical
framework of diffusion processes, and then of Mellin transforms. The closed forms obtained
are hence directly linked to the specific parametric assumptions made.
However, this model has some practical motivations as well: 
(i) the additive Brownian term can be seen as a diffusive force that is known to be
sufficient for causing {\em opinion polarization} term \cite{baccelli2015pairwise};
(ii) the power law confidence model proposed here is in a sense more natural than the  
the bounded confidence case in that it rarely accepts interactions between distant
opinions as observed in real opinion making;
(iii) in the opinion dependent interaction case, the fact that the dynamics is in continuous time  
leads to interesting new phenomena such as the possibility of a {\em fusion of opinions} 
(which does not happen in discrete time) despite the presence of the diffusive forces as
shown in the present paper. 

Whereas the main result of the present paper is the closed form for the steady state of this dynamics,
the analysis also reveals further interesting qualitative properties.
For instance, for small enough exponents, the power law confidence model forbids opinion fusions
despite interaction rates tending to infinity when the opinions get close; it also leads to {\em weak consensus}
(namely to a stabilization of the distribution of opinion differences)
despite interaction rates vanishing when the distance between opinions tends to infinity. The last
property is in contrast with the polarization of opinions which always arises in the bounded confidence case.

On the more technical side, we already mentioned the role of Mellin transforms.
Such transforms were already used to solve transport equations describing the TCP/IP
protocol \cite{hollot2001control, dumas2002markovian, baccelli2002mean, baccelli2005analysis, baccelli2007equilibria}.
When seeing instantaneous throughput as the analog of opinion distance, and packet losses as 
the analog of interactions, we get a natural connection between the halving of instantaneous
throughput in case of packet loss and the halving of opinion distance in case of agent interactions.
The main novelty regarding this literature is twofold: first, we replace the transport operator
(describing the linear increase of TCP) by a diffusion operator (representing the additive noise);
second, whereas TCP features loss rates that increase with the value of instantaneous throughput,
opinion dynamics feature interactions with a rate that {\em decreases} with opinion distance.
In spite of these differences, we could extend the mathematical machinery from one case to the other.

The structure of the paper is the following. Section \ref{sec2} describes the general framework
of the continuous-time model; this characterizes the dynamics in terms of diffusion with jumps.
It encompasses both 1) the case where the rates of interactions and jumps remain fixed and
2) the case where the rates are opinion dependent. Section \ref{sec2} further gives the main 
results and discusses their implications on opinion dynamics.
In Section 3, we discuss the notions of weak consensus and opinion fusion.
The connections between opinion dependent interactions and the bounded confidence
model are discussed in the beginning of Section \ref{sec:two_agent_model}. Subsection \ref{sec:construction}
discusses the path-wise representation of these diffusion processes with jumps.
Mellin transforms are introduced in Subsection \ref{sec3.4}. Subsections \ref{sec3.5} and \ref{sec3.6}
contain the derivations of the closed form stationary solutions. Section \ref{sec:two_agent_model} 
studies the simplest case, namely the two agent case with real-valued opinions.
We discuss various multidimensional (either more than two agents or opinions
taking their values in the Euclidean space or interactions constrained by a graph)
extensions of the basic framework in Section \ref{multi}.
In particular, we show how to leverage our analysis of both the opinion dependent and opinion independent
cases through a natural mean-field model featuring interactions between a large number of agents.

\section{Main Results on the Opinion Dynamics Model}
\label{sec2}

\subsection{Setting for the Two-Agent Problem}
The opinion of an agent at time $t$ will be denoted by $X(t)$ and is assumed to take its values
in $\mathbb{R}$ (except in Section \ref{multi} where the case of ${\mathbb R}^d$-valued
opinions is discussed). 
This assumption of continuous time and space is in part made for mathematical tractability, as already explained. 
The assumption of opinions with values in a non-compact state space is that in several papers 
\cite{blondel2010continuous, baccelli2015pairwise}.

There are two types of agents: stubborn agents, with a constant opinion, and regular agents.
In the absence of interaction, the opinion $X(t)$ of a regular agent satisfies the following
stochastic differential equation:
\begin{align}\label{eq:sde_without_interaction}
dX(t) = \mu dt + \sigma dW(t),
\end{align}
where $W(t)$ is a standard Brownian motion. We assume that the parameters $\mu$ and $\sigma$
of this diffusion are constant. In this opinion dynamics setting,
the drift parameter $\mu$ can be interpreted as the {\em bias} of the agent
and the diffusion parameter $\sigma>0$ as its {\em self-belief} coefficient.
In some cases, we assume that $\mu=0$.

We focus on pair-wise interactions, i.e., interactions following the gossip model
of \cite{mobilia2007role, acemouglu2013opinion, acemoglu2014state, baccelli2015pairwise}. 
We first solve the simplest problem, which is the two-agent case comprising one regular agent
and one stubborn agent. We consider multi-agent extensions in a second step.

The assumptions of the two-agent model are the following:
\begin{itemize}
\item {\em The stubborn agent $S$ has a fixed opinion} (say 0) at all times regardless of interactions;
\item {\em The opinion of the regular agent $X$} has a diffusive dynamics,
together with jums/updates of its opinion at each of the interaction epochs with the stubborn agent;
\item The interaction times are determined by a point process $N(t)$ with the stochastic intensity
$\Lambda(X(t))$ which is a function of the opinion difference, $X(t)$, of the two agents;
here the stochastic intensity is defined with respect to the natural filtration
of $\{X(s)\}$ (see \cite{baccelli2003}); in the power law model, the interaction rate decreases
with distance, as in, e.g., the bounded confidence model.
\item At an interaction event taking place at time $t$, the regular agent at $X(t)$ incorporates
the opinion of the stubborn agent $S$ by updating its current opinion from $X(t)$ to $X(t)/\theta$,
which is the weighted average of its opinion $X(t)$ and that of $S$, where $\theta>1$.
\item The diffusion represents the autonomous evolution of the agent's opinion in the
absence of interactions (the so-called self-beliefs of, e.g., \cite{baccelli2015pairwise}).
\end{itemize}
We broadly consider two types of interaction functions $\Lambda(X(t))$ according to the dependency of interaction clocks:
\begin{itemize}
\item{\em{Opinion independent interactions}}
: the interaction rate is constant, with $\Lambda(X(t))=\lambda>0$.
\item{\em{Opinion dependent interaction}}: the interaction rate depends on the current geometry of opinion. Several explicit results will be based on the {\em {power law model}}, where
$\Lambda(X(t))=\frac{\lambda}{|X(t)|^{\alpha}}$ for some $\alpha>0$.
\end{itemize}
One can see independent interactions as a domination type interaction where the regular
agent has to incorporate the opinion of the stubborn agent regardless of the discrepancies 
between their opinions.
One can relate the dependent interaction to {\em free will}. The free will of the regular
agent is modeled by the stochastic intensity of $\Lambda(X(t))$. Assume, for instance,
that $\Lambda(X(t))\le \lambda$. Then one can interpret $\lambda$ as the rate of interaction
offers and $\Lambda(X(t))$ as the rate of accepted interactions. In the free will example, 
the regular agent incorporates the opinion of the stubborn agent more likely if this 
opinion has some proximity with its own, and may ignore it otherwise.
The connections with the Hegselmann-Krause model \cite{hegselmann2002opinion}
are discussed in the next section.

\subsection{Related Models and Results}
When there are no random additive self-beliefs, the opinion independent interaction,
discrete-time case, was studied by Fagnani and Zampieri \cite{fagnani2008randomized}. 
The stationary solution was studied using a Markov process analysis. In the opinion dependent
interaction case, such as the bounded-confidence model
\cite{hegselmann2002opinion, como2011scaling, acemoglu2014state, baccelli2015pairwise},
there are no known explicit stationary solutions to the best of our knowledge. 
In the discrete-time case with i.i.d. additive self-beliefs,
it was shown in \cite{baccelli2015pairwise} that, for $\alpha>2$,
the dynamics with such self-beliefs is unstable. 
This means that the distribution of opinion distance is not tight,
a phenomenon that can be interpreted as {\em opinion polarization}.
If $0\le \alpha <2$, then there is a unique stationary regime for opinion distances,
a situation which is interpreted in \cite{baccelli2015pairwise} as a weak form of consensus. 

\subsection{Summary of our Results on the Two-Agent Model} 
The main analytical result of the present paper is the characterization of the stationary distribution of the {\em weak consensus} (Theorem \ref{thm:main2body_dependent}) that arises in the general power-law interaction model for all $\alpha$ in the range $0\le \alpha<1$  (which includes the opinion-independent interaction case as the special case $\alpha=0$). More precisely, the partial differential equation satisfied by the distribution of the density of the solution of the stochastic differential equation admits a unique probabilistic solution given in closed form. 

Interestingly, we can formally extend the stationary solution to this partial differential equation established to the whole range $0\le \alpha<2$
and give a probabilistic interpretation of this solution. However, in the range $1\leq \alpha <2$, the stochastic differential equation is ill-defined. The physical interpretation of this solution is unclear to the authors at this stage. We leave this interpretation as an open problem. We state these different behaviors depending on the range of $\alpha$ in Theorem \ref{thm:alpharegime}.

In the continuous-time case considered here, several interesting new phenomena appear depending on the value range of $\alpha\ge 0$. For $\alpha >2$, an accumulation point of interactions can lead to some {\em fusion of opinions}. Namely, with a positive probability, there is an infinite number of interactions in a finite time with an accumulation point at a random time $T$, and the limit of the opinion of the regular agent tends to that of the stubborn agent as $t\to T$. After this fusion time, the solution of the stochastic differential equation is ill-defined. This is why we use the term fusion rather than a strong consensus. The result of \cite{baccelli2015pairwise} also suggests that when $\alpha>2$ the polarization can also take place with positive probability. For $0\le \alpha<1$, there is no such finite accumulation point of interactions, and the solution of the stochastic differential equation exists for all times. For $1\le \alpha\le 2$, we have no understanding of the path-wise solution of the stochastic differential equation, as already mentioned.

\subsection{Summary of Results on Multi-Agent Models} 

In the opinion independent interaction case, there is no fundamental difficulty extending the
results on the two-agent problem to general multi-agent scenarios, for instance, to interactions on a graph. Depending on the types of interactions \cite{holley1975ergodic, degroot1974reaching}, the stationary distribution on general graph can be represented as a mixture of distributions or an independent sum of distributions involving hitting probability on a graph. Therefore, we shall focus on opinion dependent cases.

For the general power-law interaction case, we show in Section \ref{sec:mean_field} that our results
can be used to analyze a mean-field version of the opinion dependent interaction model.
This version features a single stubborn agent and a collection of $d$ regular agents;
all regular agents have pairwise interactions with the stubborn agent.
All regular agents have the same rate of interaction with the stubborn agent;
this rate is the negative moment of order $\alpha$ of the empirical average of the opinions of the
$d$ regular agents; each agent uses this common stochastic intensity to trigger (conditionally)
independent interactions with the stubborn agent. When $d$ tends to infinity,
we show numerically that the model tends to a mean-field limit. We also show that the properties
of the limit can be analyzed in terms of the opinion-independent model 
and a consistency equation that is shown to have a single solution for $\alpha$ in the
range $0\le \alpha <1$. Interestingly, there is no solution to the consistency equation
of this mean-field limit for $\alpha \ge 1$.

\section{The Two-Agent Stochastic Interaction Model}\label{sec:two_agent_model}

We consider the two agent model with one regular agent $X$ and one stubborn agent $S$.
At each interaction, the regular agent updates its opinion to the average of its opinion and that
of the stubborn agent. Without loss of generality, we assume that $S(t)=\mathbf{s}=0$
for all $t\geq 0$. Fix $\theta>1$ and $\theta'>1$ such that $\frac{1}{\theta} + \frac{1}{\theta'}=1$. 
Then $X(t)$ satisfies the stochastic differential equation:
\begin{align}\label{eq:dynamics}
{\mathrm d}X(t) = \mu {\mathrm d}t + \sigma {\mathrm d}W(t) - \frac{X(t-)}{\theta'}N({\mathrm d}t),
\end{align}
where $X({t-})$ is the left limit of $\{X(s)\}$ at $t$,
and $N$ is a point process on the real line with stochastic intensity
$\Lambda(X({t-}))$ (w.r.t. the natural filtration of $\{X(s)\}$ \cite{baccelli2003}).
The parameter $\theta$ (or $\theta'$) reflects the weighted average used when an interaction occurs. 

For example, if $\theta'=3$, $X(t)$ updates its opinion by weighing
$\frac{1}{\theta}=1-\frac{1}{\theta'}=\frac{2}{3}$ its own opinion and
$\frac{1}{\theta'}$ the other (stubborn) opinion;
so $X(t)=\frac{2}{3}X({t-})+\frac{1}{3}S(t)=\frac{2}{3}X({t-})$ at the moment of an interaction.
The path of $X(t)$ is almost surely continuous except at interaction times, i.e., at the epochs of $N(\cdot)$.
%
%
%

In order to discuss our analytical solution in a clear form, we will consider
three specific forms of interaction rate $\Lambda(x)$:
\begin{itemize}
\item{\textbf{(C1)}}
$\Lambda(x)=\lambda$ (opinion independent case);
\item{\textbf{(C2)}}
$\Lambda(x)=\frac{\lambda}{|x|^\alpha}$ for $\alpha>0$
(opinion dependent case, power law confidence, unbounded intensity); 
\item{\textbf{(C3)}}
$\Lambda(x)=\min\left\{\frac{\lambda}{|x|^\alpha}, L \right\}$ 
for a large constant $L>0$ 
(opinion dependent case, power law confidence, bounded intensity);
\end{itemize}
Note that {\bf (C1)} is a special case of {\bf (C2)} (when $\alpha=0$).
It is however qualitatively very different. The dependent models {\bf (C2)} and {\bf (C3)}
differ in that the former gives an unbounded intensity and the latter a bounded one.
When the intensity is unbounded, it is not guaranteed that the dynamics (\ref{eq:dynamics})
is well-defined. For instance, if $X(0)=0$, the solution of this stochastic differential equation is ill-defined. This is further discussed in Theorem \ref{thm:separation}.

\subsection{Relationship with the Bounded Confidence Model}\label{sec:relbc}
The models discussed above differ from 
the bounded confidence model \cite{hegselmann2002opinion, deffuant2000mixing}.
The latter assumes that interactions between two agents only occur when the distance between their opinions
is less than some fixed (confidence) range. In the continuous time setting, each agent 
has an exponential clock with a constant rate and when its clock ticks, it averages its opinion with that
of all opinions at a distance less than the confidence range. 
In contrast, Model {\bf (C1)} assumes a constant interaction rate, but has no confidence range limitation. 
Model {\bf (C2)} assumes a gradual decrease of the interaction rate with the opinion distance
but is never zero even for large distances.
As illustrated by the example in Figure \ref{figure:BC}, in terms of interaction rate,
model {\bf (C3)} can be seen as the closest to the bounded confidence model
as it also features a constant interaction rate within a given range.
However, like model {\bf (C2)}, model {\bf (C3)} allows for interactions at all distances.

\begin{figure}[!htb]
\centering
	\includegraphics[scale=0.4]{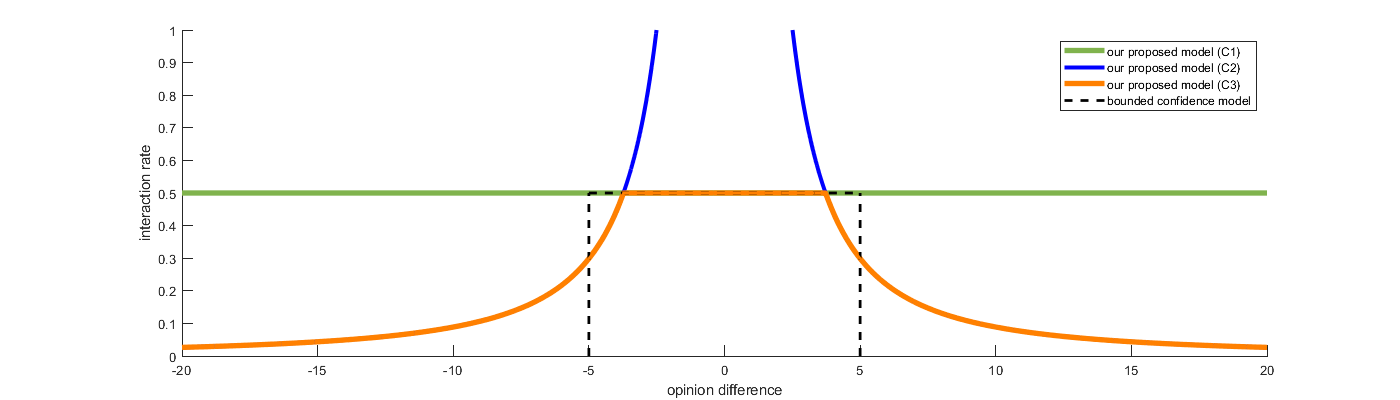} 
	\caption{The bounded confidence model with interaction rate function equal to $0.5$ on $[-5, 5]$
vs the rate functions of the proposed model. 
The parameters are $\lambda= 0.5$ for {\bf (C1)}, $\lambda = 5, \alpha=1.75$,
and $L= 0.5$ for {\bf (C2)} and {\bf(C3)}.}\label{figure:BC}
\end{figure}

In spite of the proximity of the interaction models, in the presence of Brownian self-beliefs,
the bounded confidence model and the power law confidence model differ in a fundamental way.
As we shall see below, the latter is stable (positive recurrent) for small enough values of the
interaction exponent $\alpha$, whereas the former is never stable, whatever the finite confidence range.
This last fact was already observed in \cite{baccelli2015pairwise}
for the discrete time model and immediately follows from the null recurrence of Brownian motion here.
Hence, in spite of the proximity of the interaction functions, bounded confidence always leads
to polarization whereas power law confidence models can lead to weak consensus.

\subsection{Stability and Construction of the Opinion Dependent Dynamics}\label{sec:construction}
In this section, we prove that the solution of the 
stochastic differential equation associated to model {\bf (C2)}
is well-defined when $0\leq \alpha < 1$. For this, we use
the Engelbert-Schmidt 0-1 law.
Next, we show that, almost surely, no accumulation points of interactions can appear in a finite time
horizon, which allows one to give a path-wise construction 
of the process $X(t)$ for all $t\geq 0$. We then discuss the behavior of the dynamics
when $1\leq \alpha <2$ and $\alpha\geq 2$.

\begin{lemma}[Engelbert-Schmidt 0-1's Law] \label{lem:finite_time_integral}
	Let $\{W(t)\}$ be a standard Brownian motion. Assume that $W(0)=0$. For any $t>0$, 
	\begin{align*}
	\mathbf{P}\left[\int_{0}^{t} \frac{1}{|W(s)|^\alpha}{\mathrm d}s < +\infty \quad \forall 0\leq t <\infty \right] = \begin{cases}
	    1       & \quad \text{if }\ 0\leq \alpha < 1,\\
	    0  & \quad \text{if }\ \alpha\geq 1.\\
	  \end{cases}
	\end{align*}
\end{lemma}
\begin{proof}
See \cite{karatzas2012brownian} or \cite{yen2014local}.
\end{proof}

We now show that Lemma \ref{lem:finite_time_integral}
implies the integrability of the stochastic intensity when $0\le \alpha <1$ .
The first step is:

\begin{theorem}\label{thm:separation}
Consider model {\bf (C2)} with $\mu=0$ and $0\leq \alpha <1$.
Let $X(0)=x_0$ and let $\gamma_{1}$ be the first interaction time of $X(t)$.
For all $t>0$, there exists a function $\chi(t)>0$ which does not depend on $x_0$
such that (i) $\mathbf{P}\left[ \gamma_1 \geq t \right] \geq \chi(t)$,
(ii) the map $t\to \chi(t)$ is non-increasing, and (iii) $\chi(t)\to 1$ as $t\to 0$.
\end{theorem}

\begin{proof}
Without loss of generality, we may assume that $x_0\geq 0$
(when $x_0\leq 0$, we can apply the symmetry of the Brownian motion).
From the definition of the interaction point process, 
\begin{align*}
\mathbf{P}\left[ \gamma_1 \geq t \right] =
\mathbf{E} \left[ \exp \left\{ -\int_{0}^{t} \frac{\lambda}{|W(s)+x_0|^{\alpha}}{\mathrm d}s  \right\}\right].
\end{align*}
Below, we fix $t>0$ and $\delta>0$, and consider three cases:

\noindent
\textbf{Case I.} $\left[X(0)=x_0=0 \right]$. By Lemma \ref{lem:finite_time_integral},
$\int_{0}^{t} \frac{1}{|W(s)|^{\alpha}}{\mathrm d}s < + \infty$ a.s.. Hence
\begin{align*}
\lim_{t\to 0}\int_{0}^{t} \frac{1}{|W(s)|^{\alpha}}{\mathrm d}s = 0 \quad \text{a.s.}
\end{align*}
Let $\epsilon_1(t):=\mathbf{E}\left[ \exp\left\{-\int_{0}^{t} \frac{\lambda}{|W(s)|^{\alpha}}{\mathrm d}s \right\} \right] \leq 1$.
Since $\mathbf{P}\left[\int_{0}^{t} \frac{1}{|W(s)|^\alpha}{\mathrm d}s < +\infty \right] =1$, we have
$\epsilon_1(t)>0$. 
Note that, by Dominated Convergence,
\begin{align*}
\lim_{t\to 0}\epsilon_1(t) =\mathbf{E}\left[ \exp\left\{- \lim_{t\to 0} \int_{0}^{t} \frac{\lambda}{|W(s)|^{\alpha}}{\mathrm d}s \right\} \right] = 1.
\end{align*}

\noindent
\textbf{Case II.} $\left[X(0)=x_0\geq \delta\right]$.
It is well known that for $\delta>0$, we can find $\epsilon_2(t, \delta)>0$ such that
\begin{align*}
\epsilon_2(t, \delta):=\mathbf{P}\left[ \inf_{0\leq s\leq t} W(s) \geq -\frac{\delta}{2}\right] = \Phi \left( \frac{\delta}{2\sqrt{t}} \right) - \Phi \left( -\frac{\delta}{2\sqrt{t}} \right)>0,
\end{align*}
where $\Phi(\cdot)$ is the cumulative Normal distribution function \cite{borodin2012handbook}.
So $\epsilon_2(t,\delta)> 0$ and
$\epsilon_2(t,\delta)\to 1$ as $t\to 0$.
Since $|x+x_0|\geq \frac{x_0}{2}\geq\frac{\delta}{2}$ for $x\geq -\frac{x_0}{2}$,
conditioned on $\left\{\inf_{0\leq s\leq t} W(s) \geq -\frac{\delta}{2}\right\}$
\begin{align*}
\int_{0}^{t} \frac{1}{|W(s)+x_0|^{\alpha}}{\mathrm d}s \leq \left(\frac{2}{\delta}\right)^{\alpha}t.
\end{align*}
Hence, if $x_0\ge \delta$,
\begin{eqnarray*}
\mathbf{P}\left[ \gamma_1 \geq t \right] &
=&\mathbf{E}\left[ \exp\left\{-\int_{0}^{t} \frac{\lambda}{|W(s)+x_0|^{\alpha}}{\mathrm d}s \right\} \right] \\
&\geq& \mathbf{E}\left[ \exp\left\{-\int_{0}^{t} \frac{\lambda}{|W(s)+x_0|^{\alpha}}{\mathrm d}s \right\} \bigg| \inf_{0\leq s\leq t} W(s) \geq -\frac{\delta}{2} \right] \mathbf{P}\left[ \inf_{0\leq s\leq t} W(s) \geq -\frac{\delta}{2} \right] \\
&\geq& \exp\left\{-\left(\frac{2}{\delta}\right)^{\alpha}\lambda t \right\} \epsilon_2(t, \delta):=\epsilon_3(t, \delta).
\end{eqnarray*}
So $\epsilon_3(t,\delta)> 0$.

\noindent
\textbf{Case III.} $\left[ 0<X(0)=x_0< \delta\right]$. Let $\tau_{0,\delta}$ be the first hitting time of $\{0,\delta\}$ by $W(t)$.
Then by using the Green function formula \cite[Lemma 20.10]{kallenberg2006foundations}, \cite[Lemma 5.4]{ramanan2006reflected},
\begin{align*}
\mathbf{E} \left[ \int_{0}^{\tau_{0,\delta}} \frac{\lambda}{|W(s)+x_0|^{\alpha}}{\mathrm d}s \right] = \int_{0}^{\delta} \frac{\lambda}{y^{\alpha}} g(y){\mathrm d}y,
\end{align*}
where $$g(y)=\frac{2(\min\{x_0, y \} )(\delta - \max\{x_0, y\})}{\delta}.$$
By an elementary calculation, 
\begin{align*}
\mathbf{E} \left[ \int_{0}^{\tau_{0,\delta}} \frac{\lambda}{|W(s)+x_0|^{\alpha}}{\mathrm d}s \right] =
2\lambda\left( \frac{1}{1-\alpha} -\frac{1}{2-\alpha} \right)\left( x_0\delta^{1-\alpha} - x_0^{2-\alpha}\right).
\end{align*}
This function is maximized at $x_0=\delta \left(2-\alpha\right)^{-\frac{1}{1-\alpha}}$. Hence
\begin{align*}
\mathbf{E} \left[ \int_{0}^{\tau_{0,\delta}} \frac{\lambda}{|W(s)+x_0|^{\alpha}}{\mathrm d}s \right] \leq 2\lambda\left( \frac{1}{1-\alpha} -\frac{1}{2-\alpha} \right) \left( (2-\alpha)^{-\frac{1}{1-\alpha}}-(2-\alpha)^{-\frac{2-\alpha}{1-\alpha}}\right)\delta^{2-\alpha} =: \epsilon_4(\delta).
\end{align*}
Notice that for $\delta$ small enough, $\epsilon_4(\delta)<1$.
Then we consider two sub-cases, depending on the order of $t$ and $\tau_{0,\delta}$. 

When $t\geq \tau_{0,\delta}$,
\begin{align*}
\int_{0}^{t} \frac{1}{|W(s)+x_0|^{\alpha}}{\mathrm d}s = \int_{0}^{\tau_{0,\delta}} \frac{1}{|W(s)+x_0|^{\alpha}}{\mathrm d}s + \int_{\tau_{0,\delta}}^{t} \frac{1}{|W(s)+x_0|^{\alpha}}{\mathrm d}s =:\xi.
\end{align*}
By the Strong Markov property of Brownian motion, we can rewrite
\begin{equation*}
\xi = \int_{0}^{\tau_{0,\delta}} \frac{1}{|W(s)+x_0|^{\alpha}}{\mathrm d}s + \int_{0}^{t-\tau_{0,\delta}} \frac{1}{|W'(s)+z'|^{\alpha}}{\mathrm d}s
\leq \int_{0}^{\tau_{0,\delta}} \frac{1}{|W(s)+x_0|^{\alpha}}{\mathrm d}s + \int_{0}^{t} \frac{1}{|W'(s)+z'|^{\alpha}}{\mathrm d}s,
\end{equation*}
where $z'=0$ if $W_{\tau_{0,\delta}}$ is $0$ and $z'=\delta$ otherwise.
Here, $W'(t)$ is an independent Brownian motion with $W_0'=0$.

When $ t<\tau_{0,\delta}$,
\begin{align*}
\int_{0}^{t} \frac{1}{|W(s)+x_0|^{\alpha}}{\mathrm d}s \leq \int_{0}^{\tau_{0,\delta}} \frac{1}{|W(s)+x_0|^{\alpha}}{\mathrm d}s.
\end{align*}
Therefore in both sub-cases,
\begin{align*}
\int_{0}^{t}\frac{1}{|W(s)+x_0|^{\alpha}}{\mathrm d}s \leq \int_{0}^{\tau_{0,\delta}} \frac{1}{|W(s)+x_0|^{\alpha}}{\mathrm d}s + \int_{0}^{t} \frac{1}{|W'(s)+z'|^{\alpha}}{\mathrm d}s .
\end{align*}
We can summarize Case III by
\begin{eqnarray*}
\mathbf{P}\left[ \gamma_1 \geq t\right] & = &
\mathbf{E}\left[ \exp\left\{-\int_{0}^{t} \frac{\lambda}{|W(s)+x_0|^{\alpha}}{\mathrm d}s \right\} \right]\\ 
&\geq & \mathbf{E}\left[ \exp\left\{ -\int_{0}^{\tau_{0,\delta}} \frac{\lambda}{|W(s)+x_0|^{\alpha}}{\mathrm d}s \right\} \right]\mathbf{E}\left[ \exp\left\{ - \int_{0}^{t} \frac{\lambda}{|W'(s)+z'|^{\alpha}}{\mathrm d}s \right\} \right]\\
&\geq & \left( 1-  \epsilon_4(\delta) \right) \min\left\{ \epsilon_1(t), \epsilon_3(t,\delta) \right\} =:\epsilon_5(t,\delta),
\end{eqnarray*}
where we used $\mathbf{E}\left[ e^{-Y} \right] \geq 1-\mathbf{E}\left[Y\right]$ for the last inequality. Putting all three cases together, we conclude that
\begin{align*}
\mathbf{P}\left[ \gamma_1 \geq t \right] \geq \min\{ \epsilon_1(t),\epsilon_3(t,\delta),\epsilon_5(t,\delta) \} =:
\widetilde \epsilon(t,\delta) >0,
\end{align*}
where $\widetilde \epsilon(t,\delta)$ does not depend on $x_0$ and is strictly positive if $\delta$ is small enough.

Let
$$ \chi(t)= 
\begin{cases}
\widetilde \epsilon~(t,t^{1/3}) & \text{ if } \epsilon_4(t^{1/3}) <1\\
1 & otherwise.
\end{cases}
$$
Therefore, we conclude that $\chi(t)$ is strictly positive, non-increasing, and that in addition
$\chi(t)\to 1$ as $T\to 0$.
\end{proof}


\begin{corollary}\label{cor:unif_sep}
Under the assumptions of Theorem \ref{thm:separation},
the random variable $\gamma_1$ is stochastically larger than
a random variable $\eta$ with a distribution $H(t):=1-\chi(t)$ on $\mathbb{R}^{+}_0$ which
does not depend on $x_0$. 
\end{corollary}
\begin{proof}
The function $1-\chi(t)$ constructed in the theorem can be taken as the cumulative distribution function
of a random variable $\eta$ on $\mathbb{R}^+\cup \{\infty\}$.
The properties established in the theorem show that (1) $\gamma_1$ is
stochastically larger than $\eta$, (2) the distribution of $\eta$ does not
depend on $x_0$, (3) $\eta$ is strictly positive a.s.
\end{proof}

When $0\leq \alpha<1$, one can build the process in a path-wise sense by induction
on the stopping times $\gamma_n$, $n=1,2,\ldots$ of interactions, the general idea is that the time that elapses between $\gamma_n$ and the first interaction time after
$\gamma_n$ is stochastically larger than a random variable $\eta_n$ with distribution $H(t)$. 
This plus the strong Markov property then imply that $\gamma_n\to \infty$ a.s. as $n\to \infty$.
More precisely, assume that $W(0)=0$. Let $\gamma_0=0$ and $\gamma_1$ be the first interaction time.
Let $\mathcal{F}_t$ be the filtration of $W(t)$ and $N(t)$. Conditioning on $\left\{ X(0)=x_0 \right\}$,
$X(t)=W(t)+x_0$ for all $0\leq t<\gamma_1$, and $\gamma_1$ is an $\mathcal{F}(t)$-stopping time.
Theorem \ref{thm:separation} implies that $\gamma_1$ is a strictly positive random variable a.s.
In addition, there exists a random variable $\eta_1$ with distribution $H$ such that
$\gamma_1\ge \eta_1$ a.s.
On the event $\gamma_1<+\infty$, we define $X({\gamma_1}):=\frac{W({\gamma_1})+x_0}{\theta}$.
Again conditioning on $\left\{ W({\gamma_1})=x_1 \right\}$, by the strong Markov property,
$W({\gamma_1+t})\,{\buildrel d \over =}\, W(t) + x_1$.
Then there exists $\gamma_2>\gamma_1$ and we can define 
\begin{align*}
X(\gamma_1+t) := X({\gamma_1}) + W({\gamma_1+t}) - W({\gamma_1})=\frac{x_0}{\theta}-\frac{x_1}{\theta’}+W({\gamma_1+t}),
\end{align*}
for all $0\leq t<\gamma_2-\gamma_1$. By the same argument as above, there exists a random variable
$\eta_2$ with distribution $H$ such that $\gamma_2-\gamma_1\ge \eta_2$ a.s. 
By the strong Markov property, we can take $\eta_2$ independent of $\eta_1$. 
More generally, one proves by induction the existence of the stopping times
$\gamma_1<\gamma_2<\cdots$ and i.i.d. random variables $\eta_1,\eta_2,\ldots$ 
such that one can construct $X(\gamma_n+t)$ by the formula 
\begin{align*}
X(\gamma_n+t) := X({\gamma_n})+W({\gamma_n+t})-W({\gamma_n}),
\end{align*}
for $0\leq t < \gamma_{n+1}-\gamma_{n}$ and
\begin{align*}
X({\gamma_{n+1}}):=\frac{1}{\theta}X({\gamma_{n+1}-})=\lim_{t\to (\gamma_{n+1}-\gamma_{n})} \frac{1}{\theta} X(\gamma_n+t).
\end{align*}
There remains to prove that
$\gamma_n$ tends to infinity with $n\to \infty$.
Since the sequence $\{\eta_n\}$ is i.i.d., by the strong law of large numbers
\begin{align*}
\gamma_n \ge \sum_{i=0}^{n-1} \eta_i \to \infty \quad\text{as $n\to\infty$}.
\end{align*}
Therefore, the process $X(t)$ is a.s. well-defined for all times $t$.

\begin{theorem}\label{thm:alpharegime}
Assume that $\Lambda(X(t))=\frac{\lambda}{|X(t)|^{\alpha}}$. For $0\leq \alpha<1$,
there is no finite accumulation time point of interactions almost surely
and the stochastic process $\{X(t)\}$ is well-defined over the whole time horizon.
\end{theorem}

When $\alpha \geq 1$, the proof of Theorem \ref{thm:separation} cannot be adapted.
For instance, the dynamics are always ill-defined when starting from $x_0=0$ since,
by Lemma \ref{lem:finite_time_integral}, for all $t>0$,
\begin{align*}
\mathbf{P}\left[\int_{0}^{t} \frac{\lambda}{|W(s)|^{\alpha}}{\mathrm d}s = +\infty\right] =1.
\end{align*}
For $1\leq \alpha < 2$, the process has no finite accumulation point of interactions almost surely until the first
hitting time of 0 and is path-wise ill-defined after the hitting time. For $\alpha \geq 2$, the dynamics is ill defined even when starting from $x_0\ne 0$ by the law of the iterated logarithm for Brownian motion. See Theorem 5.1 and Corollary 5.3 of \cite{morters2010brownian}. Note that when $x_0=0$, the process is ill-defined by Lemma \ref{lem:finite_time_integral}. The law of the iterated logarithm for Brownian motion implies for any $t>0$, there exists a constant $C>1$ such that $|W(t+h)-W(t)|\leq C|2h\log\log(1/h)|^{\frac{1}{2}}$ almost surely for all $0\leq h \leq \epsilon$ with some $\epsilon>0$. Then for $\alpha\geq 2$ and given the first hitting time $\gamma_1$, \begin{align*}
    |W(\gamma_1)+x_0-W(\gamma_1-h)-x_0| = |W(\gamma_1-h)+x_0| \leq C|2h\log\log(1/h)|^{\frac{1}{2}}.
\end{align*}
Therefore when $x_0\neq 0$, by choosing $\epsilon'<\min\{\epsilon,\gamma_1, 1/e \}$ (where $e$ is Euler constant),
\begin{align*}
    \int_{0}^{\gamma_1} \frac{\lambda}{|W(s)+x_0|^{\alpha}}{\mathrm d}s &\geq  \int_{\gamma_1-\epsilon'}^{\gamma_1} \frac{\lambda}{C^{\alpha}|2(s-\gamma_1)\log\log (1/(s-\gamma_1))|^{\frac{\alpha}{2}}}{\mathrm d}s\\
    &\geq  \int_{0}^{\epsilon'} \frac{\lambda}{C^2|2s\log\log(1/s)|}{\mathrm d}s = +\infty \quad \text{almost surely.}
\end{align*}
Hence when $\alpha\geq 2$, finite accumulations of interactions occur almost surely and the process is path-wise ill-defined in the vicinity of zero.


\subsection{Fokker-Planck Evolution Equation}\label{sec:fpe}

We now establish the Kolmogorov forward equation (also referred to as the Fokker-Planck evolution equation)
of the probability density of $X(t)$. This will be done under the following assumptions
(${\bf H}$):
\begin{itemize}
\item[(i)] the function $x\to \Lambda(x)$ is measurable,
\item[(ii)] for all $0\leq a<b<+\infty$, $\int_{a}^{b} \Lambda(X(t)) {\mathrm d}t < +\infty$ almost surely.
\end{itemize}
Then $\Lambda(X(t))$ is predictable and $X(t)$ under \eqref{eq:dynamics} satisfies \cite[Assumption 6.1.1]{bjork2011introduction}. Note that conditions {\bf (C1)} and {\bf (C3)} imply ${\bf H}$.
Under {\bf (C2)}, ${\bf H}$ holds when $0\le \alpha <1$. For the next theorem, we use
the smoothness of the density of $X(t)$ proved in Appendix \ref{sec:appendix1}.

\begin{theorem}\label{thm:evolution}
Assume $\theta>1$. Under ${\bf H}$, the density $p(t,x)$ of $X(t)$ 
satisfies the non-local partial differential equation :
\begin{align}\label{eq_pde}
		\frac{\partial p(t,x)}{\partial t} = \frac{\sigma^2}{2}\frac{\partial^2 p(t,x)}{\partial x^2} - \mu \frac{\partial p(t,x)}{\partial x}  - \Lambda(x) p(t,x) +\theta \Lambda(\theta x) p_t(\theta x).
	\end{align}
\end{theorem}
\begin{proof}
We follow the approach described by Bj\"{o}rk \cite[Proposition 6.2.1, Proposition 6.2.2]{bjork2011introduction}.
Under {\bf H}, the stochastic intensity $\Lambda(X(t-))$ is locally integrable and predictable
in the sense of \cite{bjork2011introduction}.
For any function $g:\mathbb{R}\times\mathbb{R}\to \mathbb{R}$
which is in $C^{1,2}$, we have, from Ito's formula:
\begin{align*}
{\mathrm d}g(t,X(t)) =& \left\{\frac{\partial g}{\partial t}(t,X(t))+\mu\frac{\partial g}{\partial x}(t,X(t)) + \frac{\sigma^2}{2}\frac{\partial^2 g_t}{\partial x^2}(t,X(t)) \right\}{\mathrm d}t + \sigma\frac{\partial g}{\partial x}(t,X(t)){\mathrm d}W(t)\\
& + \left(g\left(t,\frac{X(t)}{\theta}\right) - g(t,X(t)) \right)N({\mathrm d}t).
\end{align*}
Then the infinitesimal generator can be described as follows. 
For any function $f:\mathbb{R}\to \mathbb{R}$ which is in $C^{2}$, we have
\begin{align*}
\mathcal{A}f = \mu\frac{\mathrm{d} f}{\mathrm{d} x}(x) + \frac{\sigma^2}{2}\frac{\mathrm{d}^2 f}{\mathrm{d} x^2}(x)
+ \left(f\left(\frac{1}{\theta}x\right)- f(x)\right)\Lambda(x).
\end{align*}
The adjoint operator $\mathcal{A}^*$ is given by
\begin{align*}
\mathcal{A}^*f = -\mu\frac{\mathrm{d} f}{\mathrm{d} x}(x) +
\frac{\sigma^2}{2}\frac{\mathrm{d}^2 f}{\mathrm{d} x^2}(x) +\theta  f\left(\theta x\right)\Lambda\left(\theta x\right)- f(x)\Lambda(x),
\end{align*}
since $\int f\left(\frac{1}{\theta}x\right)h(x) {\mathrm d}x = \int \theta f(x)h(\theta x){\mathrm d}x$ for all $h$. 
By Lemma \ref{lem:smooth_density} in Appendix \ref{sec:appendix1}, the probability density function $p(t,x) \in C^{1,2}$. So the probability density function $p(t,x)$
satisfies $\mathcal{A}^*p(t,x) = \frac{\partial p(t,x)}{\partial t} $.
Therefore we have the forward evolution equation:
\begin{align*}
\frac{\partial p(t,x)}{\partial t} = \frac{\sigma^2}{2}\frac{\partial^2 p(t,x)}{\partial x^2} - \mu \frac{\partial p(t,x)}{\partial x}  - \Lambda(x) p(t,x) +\theta \Lambda(\theta x) p_t(\theta x).
\end{align*}
\end{proof}
In steady state, the density $p(t,x)$ would be invariant over the time $t$.
Therefore $\frac{\partial p(t,x)}{\partial t}=0$.
\begin{corollary}
The stationary distribution, when it exists, satisfies the non-local ordinary differential equation (ODE):
\begin{align}\label{eq_ode}
\sigma^2\frac{{\mathrm d}^2 p(x)}{{\mathrm d} x^2} - 2\mu \frac{{\mathrm d} p(x)}{{\mathrm d} x}  = 2\Lambda(x) p(x) -2\theta \Lambda(\theta x) p(\theta x).
	\end{align}
\end{corollary}


\subsection{Some Tools}
\label{sec3.4}
To find the stationary solution of the stochastic differential equation \eqref{eq:dynamics}
or the ordinary differential equation \eqref{eq_ode}, we will rely on 1) PASTA and 2) Mellin transforms.

\subsubsection{PASTA}
The Poisson Arrivals See Time Averages (PASTA) property of stochastic processes
is well-known in Queuing theory \cite{wolff1982poisson, melamed1990arrivals, melamed1990arrivals2, baccelli2003}.
Let $N(t)$ be a stationary point process with points $\{t_n\}$.
Let $\{\mathcal{F}_t\}$ a filtration such that $N(t)$ is ${\mathcal F}_t$-measurable for all $t$.

\begin{lemma}\label{lem:pasta}
Let $X(t)$ be a $\mathcal{F}_t$-predictable, stationary stochastic process. The sequence
\begin{align*}
Y(n) := X(t_n-), \quad n= \ldots,-2,-1,0,1,2,\ldots
\end{align*}
is stationary.
If the $\mathcal{F}_t$ stochastic intensity of $N$ is constant, 
then the stationary distribution of $Y(n)$ coincides with the stationary distribution of $X(t)$. 
\end{lemma}
\begin{proof}
See \cite[Theorem 3.3.1]{baccelli2003}.
\end{proof}

\subsubsection{Mellin Transform}
The Mellin transform of a non-negative function $f(x)$ on $\mathbb{R}_+=(0,\infty)$ is defined by
\begin{align}\label{eq:mellin}
\mathcal{M}(f;s)=\int_{0}^{\infty} x^{s-1}f(x){\mathrm d}x,
\end{align}
when the integral exists. So the Mellin transform is an extended moment transform of a function $f(x)$.

The integral \eqref{eq:mellin} defines a transform in a vertical strip of the complex $s$ plane. 
Assuming that $\mathcal{M}(f;s)$ is finite for $a<\text{Re}(s)<b$,
the inversion of the Mellin transform is given by
\begin{align*}
f(x)=\frac{1}{2\pi i}\int_{c-\infty i}^{c+\infty i} x^{-s}\mathcal{M}(f;s){\mathrm d}s \quad \text{for } a<c<b.
\end{align*}

We will also leverage the following Euler type identity \cite{baccelli2007equilibria}:
\begin{lemma}\label{lem:infmulsum}
	For any $\theta> 1$,
	\begin{align*}
	\prod_{k=0}^{\infty}\left( 1-\frac{1}{\theta^{s+k}} \right) = \sum_{n= 0}^{\infty} \frac{1}{\theta^{sn}}\prod_{k=1}^{n} \left( \frac{\theta}{1-\theta^k} \right),
	\end{align*}
where by convention $\prod_{k=1}^{0} \left( \frac{\theta}{1-\theta^k} \right) = 1$.
\end{lemma}

\subsection{Probabilistic Representation of the Solution under {\bf (C1)}}
\label{sec3.5}
In this section, we assume that {\bf (C1)} holds and that the bias term $\mu$ in \eqref{eq:dynamics}
can be non-zero. A sample path is plotted in Figure \ref{fig:evolution1} for illustration.
\begin{figure}[h]
\centering
	\includegraphics[scale=0.3]{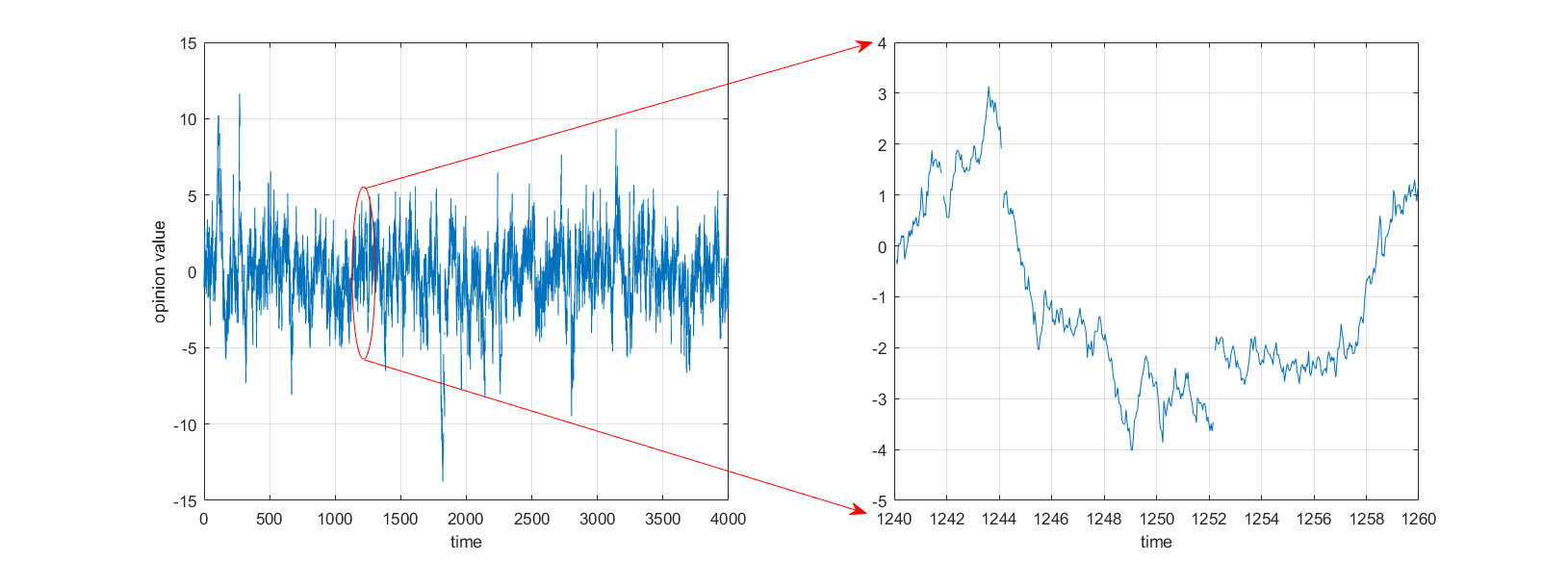} 
	\caption{Evolution of the opinion value when $\mu=0$, $\lambda = 2.0$, $\sigma=3.0$, $\alpha=0$, and $\theta=\theta'=2.0$}
	\label{fig:evolution1}
\end{figure}


Let $T=\{t_1,t_2,\cdots\}$ where $t_1\leq t_2\leq \cdots$ denotes the set of epochs of
the interaction Poisson point process of intensity $\lambda$. At time $t_n$, the regular agent
$X$ interacts with the stubborn agent $S$. For each $n$, let $Y(n) :=  \lim_{t\uparrow t_n} X(t)= X({t_n^{-}})$,
denote the state just prior to the interaction time $t_n$.
Let $\Delta t_n := t_{n+1}-t_{n}$.
The sequence $\{\Delta t_n\}$ is an i.i.d. sequence of Exponential$(\lambda)$ random variables.
The following stochastic recurrence equation holds:
\begin{align}\label{eq:stochasticeq}
Y({n+1}) = \frac{1}{\theta} Y(n) + W'(n),
\end{align}
where the sequence $\{W'(n)\}$ is again an i.i.d. sequence of  
random variables with density
$$ h(x)= \int_{0}^\infty 
\frac 1 {\sqrt{2\pi \sigma^2 t}} e^{-\frac{(x-\mu t)^2}{2 \sigma^2 t}}
\lambda e^{-\lambda t} {\mathrm d}t.$$
The sequence $\{Y(n)\}$ is called an embedded chain of $X(t)$. 
From the recurrence equation \eqref{eq:stochasticeq}, we can 
represent the stationary solution $Y$ as
\begin{align}\label{eq:stochastic_solution}
Y = \sum_{n=0}^{\infty} \frac{W'(n)}{\theta^n}.
\end{align}
Consider a stationary Poisson point process with i.i.d. marks.
The mark of point $t_n$ is a standard Brownian motion starting from 0 (only the restriction
of this Brownian motion from time 0 to time $t_{n+1}-t_n$ is useful).
Let $\{{\mathcal F}_t\}$ be the sigma algebra generated
by this marked point process. The stochastic process $X(t)$ is ${\mathcal F}_t$
adapted and also ${\mathcal F}_t$-predictable when assuming its paths are left-continuous.
The ${\mathcal F}_t$-intensity of $N$ is the constant $\lambda$.
It then follows from the PASTA property in Lemma \ref{lem:pasta} that the stationary distribution of $X({t})$
coincides with the distribution of $Y$. Hence, the stationary distribution of $X$ is
a geometric sum of i.i.d. mixtures of Gaussian random variables.
\begin{proposition}\label{pro:probilistic_solution}
Under {\bf (C1)} the stationary distribution of $X$ is that of the sum 
\begin{align}\label{eq:stochasticsolution}
\sum_{j=0}^{\infty} \frac{V({j})}{\theta^j},
\end{align}
where the $V(j)$'s are i.i.d. mixtures of Gaussians with 
density $h$.
\end{proposition}
\begin{figure}[h]
\centering
	\includegraphics[scale=0.3]{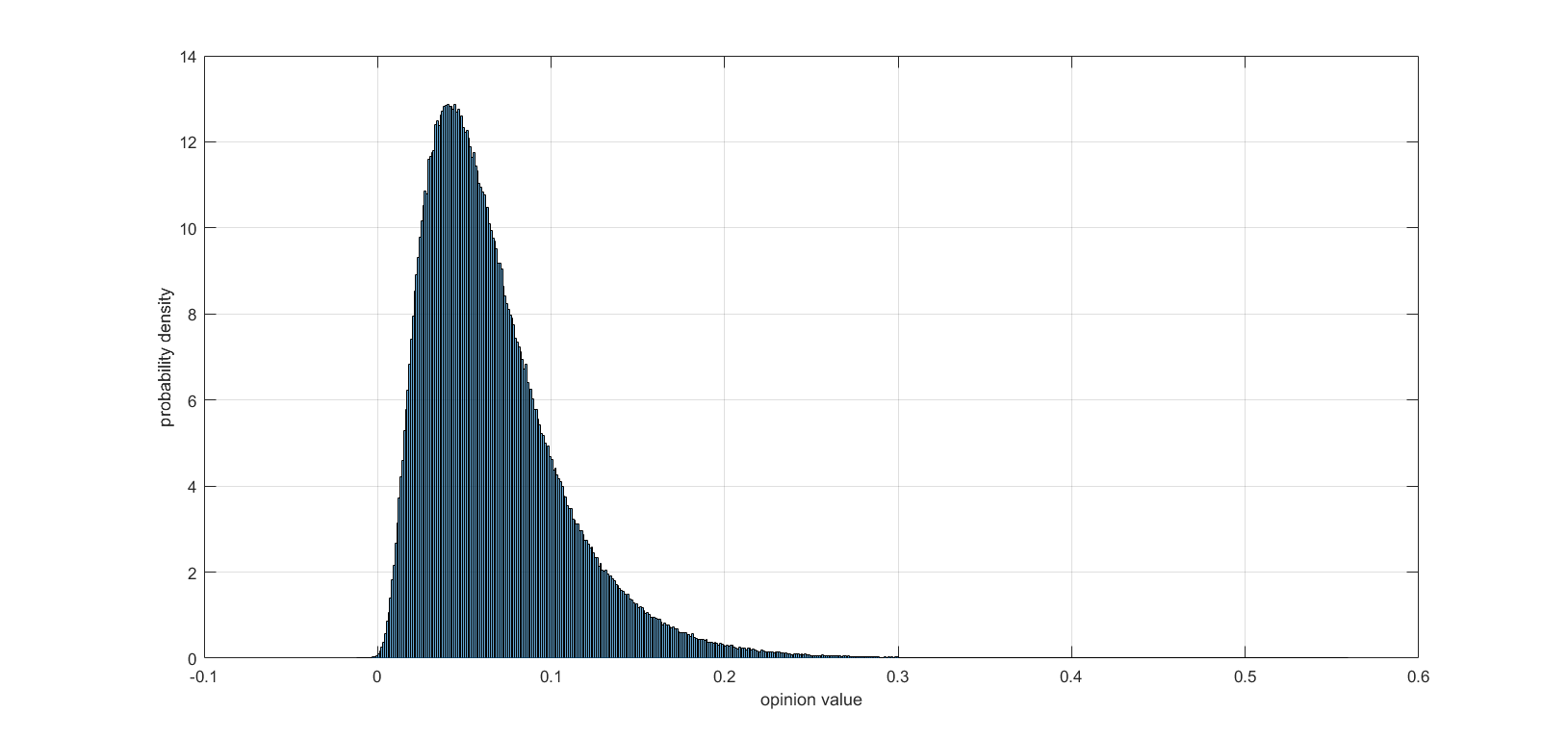} 
	\caption{Simulated histogram when $\lambda = 3.0$, $\sigma=0.02$, $\mu=0.1$, $\alpha=0$, and $\theta=\theta'=2.0$}
	\label{fig:evolution10}
\end{figure}
\begin{corollary}
The characteristic function of the stationary distribution of $X$ is
\begin{align*}
\mathbf{E}\left[e^{i\xi X({+\infty})}\right] = \prod_{j=0}^{\infty}\left[ \frac{\lambda}{\lambda-\frac{i\mu \xi}{\theta^j} +\frac{\sigma^2 \xi^2}{2\theta^{2j}}} \right].
\end{align*}
\end{corollary}
Note that each of the summands in
\eqref{eq:stochasticsolution} (before scaling by $\theta^j$) follows an i.i.d. geometric
stable distribution (Linnik distribution), and that the stationary distribution is a
geometric sum of i.i.d. geometric stable random variables. Figure \ref{fig:evolution10} 
plots the density of this distribution. 

\subsection{Analytical Solution of {\bf {(C2)}} without Bias Term}
In this section, we consider the case {\bf (C2)} under the assumption that the bias term is 0.
We solve the ordinary differential equation \eqref{eq_ode} by leveraging Mellin transforms.
While the definition of the process can only be granted when $0\leq \alpha <1$,
we nevertheless consider the case $0\leq \alpha < 2$ when solving Equation \eqref{eq_ode}.

For comparison to the opinion independent dynamics, we plot a sample of the opinion dependent
process in Figure \ref{fig:evolution3}. The main qualitative difference with the opinion dependent
process is that the interaction is more frequent around the zero opinion value (the stubborn agent opinion),
so that the process can hardly escape the vicinity of the stubborn agent opinion.
\begin{figure}[h]
\centering
	\includegraphics[scale=0.3]{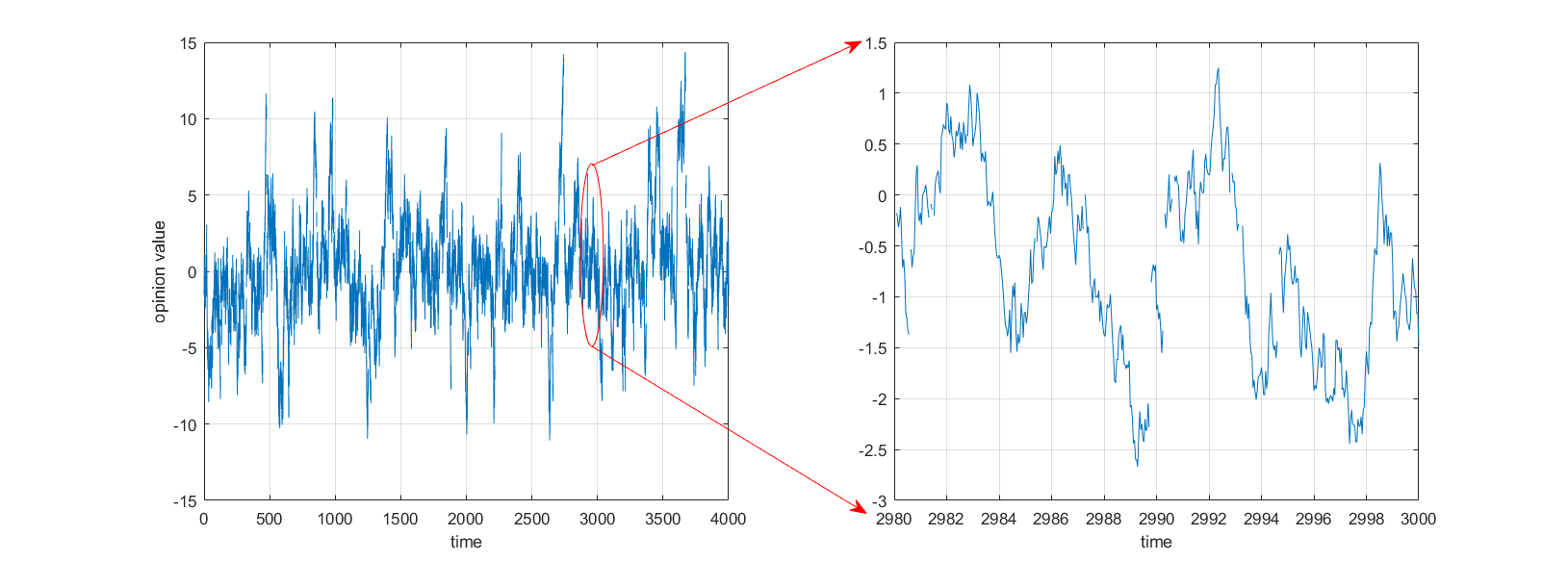} 
	\caption{Evolution of the opinion when $\lambda =2.0$, $\sigma=3.0$, $\alpha=0.5$, and $\theta=\theta'=2.0$}
	\label{fig:evolution3}
\end{figure}

\begin{theorem}\label{thm:main2body_dependent}
Consider case {\bf {(C2)}} with $\theta>1$, $\mu=0$, and $0\leq \alpha <2$. 
The unique density $p(x)$ in $\mathbb R$ which is $C^2$ and solution 
of the ordinary differential equation \eqref{eq_ode} is
	\begin{align*} 
	p(x)=2\phi (2-\alpha) \sum_{n=0}^{\infty} \frac{a_n}{\theta^{n}} \sqrt{\left(\frac{2\lambda }{\sigma^2(2-\alpha)^2}\right)^{\frac{1}{2-\alpha}} \theta^{n} |x| }\mathrm{BesselK}\left(\frac{1}{2-\alpha},2\sqrt{\frac{2\lambda \theta^{n(2-\alpha)}|x|^{2-\alpha}}{\sigma^2(2-\alpha)^2}} \right),
	\end{align*}
where
$$
a_0=1,\quad
a_n=\prod_{k=1}^{n} \left(\frac{\theta^{2-\alpha}}{1-\theta^{k(2-\alpha)}} \right),\quad \quad
\phi = \frac{\left( \frac{2\lambda}{\sigma^2(2-\alpha)^2} \right)^{\frac{1}{2-\alpha}}}{2\Gamma\left(\frac{1}{2-\alpha}\right)\Gamma\left(\frac{2}{2-\alpha}\right)} \left(\prod_{k=0}^{\infty}\left( 1- \frac{1}{\theta^{2+k(2-\alpha)}} \right)\right)^{-1},$$ 
and
$$\mathrm{BesselK}(\nu,z)=\frac{\Gamma\left(\nu+\frac{1}{2}\right)(2z)^\nu}{\sqrt{\pi}}\int_{0}^{\infty}\frac{\cos t}{(t^2+z^2)^{\nu+\frac{1}{2}}}{\mathrm d}t.$$
\end{theorem}
\begin{proof}
	We divide $p(x)$ into two components. $p(x)=p_+(x)+p_-(x)$ where $p_+(x)=p(x)\mathbf{1}_{\{x\geq 0 \}}$ and $p_-(x)=p(x)\mathbf{1}_{\{x< 0 \}}$. It is easy to see that each component satisfies the same equation and $p_+(x)=p_-(-x)$ for $x>0$. So by symmetry, it suffices to solve the equation for $p_+(x)$, that is
	\begin{align*}
		\sigma^2x^{\alpha}\frac{d^2 p_+(x)}{d x^2} = 2\lambda p_+(x) -2\theta^{1-\alpha}\lambda p_+(\theta x).
	\end{align*}
	Let $M(s):=\mathcal{M}\left( p_+(x) ;s \right)$, where $\mathcal{M}\left(p_+(x);s\right)$ is Mellin transform with respect to $s$. By transforming the equation, we have
	\begin{align*}
		\sigma^2 (s+\alpha)(s+1+\alpha)M(s+\alpha)=2\lambda \left(1-\frac{1}{\theta^{s+1+\alpha}}\right) M(s+2).
	\end{align*}
	Let $M(s)=f(s)\Gamma\left( \frac{s}{2-\alpha} \right)\Gamma\left( \frac{s+1}{2-\alpha}\right)\left(\frac{\sigma^2(2-\alpha)^2}{2\lambda}\right)^{\frac{s}{2-\alpha}}$ with introducing $f(s)$ another yet analytic function. The above equation can be re-written as
	\begin{align*}
		f(s)=\left( 1-\frac{1}{\theta^{s+1}}\right) f(s+2-\alpha).
	\end{align*}
	Since $2-\alpha>0$, we may expand $f(s)$ toward increasing direction of $s$ into infinitely many times. By introducing an indefinite constant $\phi>0$, $f(s)$ can have the form
	\begin{align*}
		f(s) = \phi \prod_{k=0}^{\infty} \left( 1- \frac{1}{\theta^{s+1+k(2-\alpha)}} \right).
	\end{align*}
	By plugging $f(s)$ into $M(s)$ above, we have the form of $M(s)$
	\begin{align}\label{eq_mellin}
		M(s)=\phi\Gamma\left( \frac{s}{2-\alpha} \right)\Gamma\left( \frac{s+1}{2-\alpha}\right)\left(\frac{\sigma^2(2-\alpha)^2}{2\lambda}\right)^{\frac{s}{2-\alpha}} \prod_{k=0}^{\infty} \left( 1- \frac{1}{\theta^{s+1+k(2-\alpha)}} \right).
	\end{align}
	Since $p(x)$ is a probability density and $p_+(x)=p_-(-x)$ for $x>0$, $M(1)=\frac{1}{2}$. This implies
	\begin{align*}
		\phi = \frac{\left( \frac{2\lambda}{\sigma^2(2-\alpha)^2} \right)^{\frac{1}{2-\alpha}}}{2\Gamma\left(\frac{1}{2-\alpha}\right)\Gamma\left(\frac{2}{2-\alpha}\right)} \left(\prod_{k=0}^{\infty}\left( 1- \frac{1}{\theta^{2+k(2-\alpha)}} \right)\right)^{-1}.
	\end{align*}
	Applying Lemma \ref{lem:infmulsum}, the inverse Mellin transform, to \eqref{eq_mellin}, and the residue theorem guides us to reach the solution.
	\begin{align*}
		&p_+(x)\\
		=&\frac{1}{2\pi i}\int_{c-i\infty}^{c+i\infty} x^{-s}\phi\Gamma\left( \frac{s}{2-\alpha} \right)\Gamma\left( \frac{s+1}{2-\alpha}\right)\left(\frac{\sigma^2(2-\alpha)^2}{2\lambda}\right)^{\frac{s}{2-\alpha}} \prod_{k=0}^{\infty} \left( 1- \frac{1}{\theta^{s+1+k(2-\alpha)}} \right) ds\\
		=&\frac{1}{2\pi i}\int_{c-i\infty}^{c+i\infty} x^{-s}\phi\Gamma\left( \frac{s}{2-\alpha} \right)\Gamma\left( \frac{s+1}{2-\alpha}\right)\left(\frac{\sigma^2(2-\alpha)^2}{2\lambda}\right)^{\frac{s}{2-\alpha}} \sum_{n=0}^{\infty} \frac{1}{\theta^{n\left(s+1\right)}}\prod_{k=1}^{n} \left(\frac{\theta^{2-\alpha}}{1-\theta^{k(2-\alpha)}} \right) ds\\
		=&\sum_{n=0}^{\infty} \frac{2\phi (2-\alpha)a_n}{\theta^{n}} \sqrt{\left(\frac{2\lambda }{\sigma^2(2-\alpha)^2}\right)^{\frac{1}{2-\alpha}} \theta^{n} |x| }\text{BesselK}\left(\frac{1}{2-\alpha},2\sqrt{\frac{2\lambda \theta^{n(2-\alpha)}|x|^{2-\alpha}}{\sigma^2(2-\alpha)^2}} \right),
	\end{align*}
	where $a_n=\prod_{k=1}^{n} \left(\frac{\theta^{2-\alpha}}{1-\theta^{k(2-\alpha)}} \right)$. We note that we use the change of variable by $s'=\frac{s}{2-\alpha}$ at the step of the inverse Mellin transform, and the following observation for $a>0$:
	\begin{align*}
	\frac{1}{2\pi i}\int_{c-i\infty}^{c+i\infty} x^{-s}\Gamma(s)\Gamma(s+a)ds = 2x^{\frac{1}{2}a}\text{BesselK}\left(a,2\sqrt{x}\right).
	\end{align*}
\end{proof}
\begin{figure}[!htb]
\centering
	\includegraphics[scale=0.3]{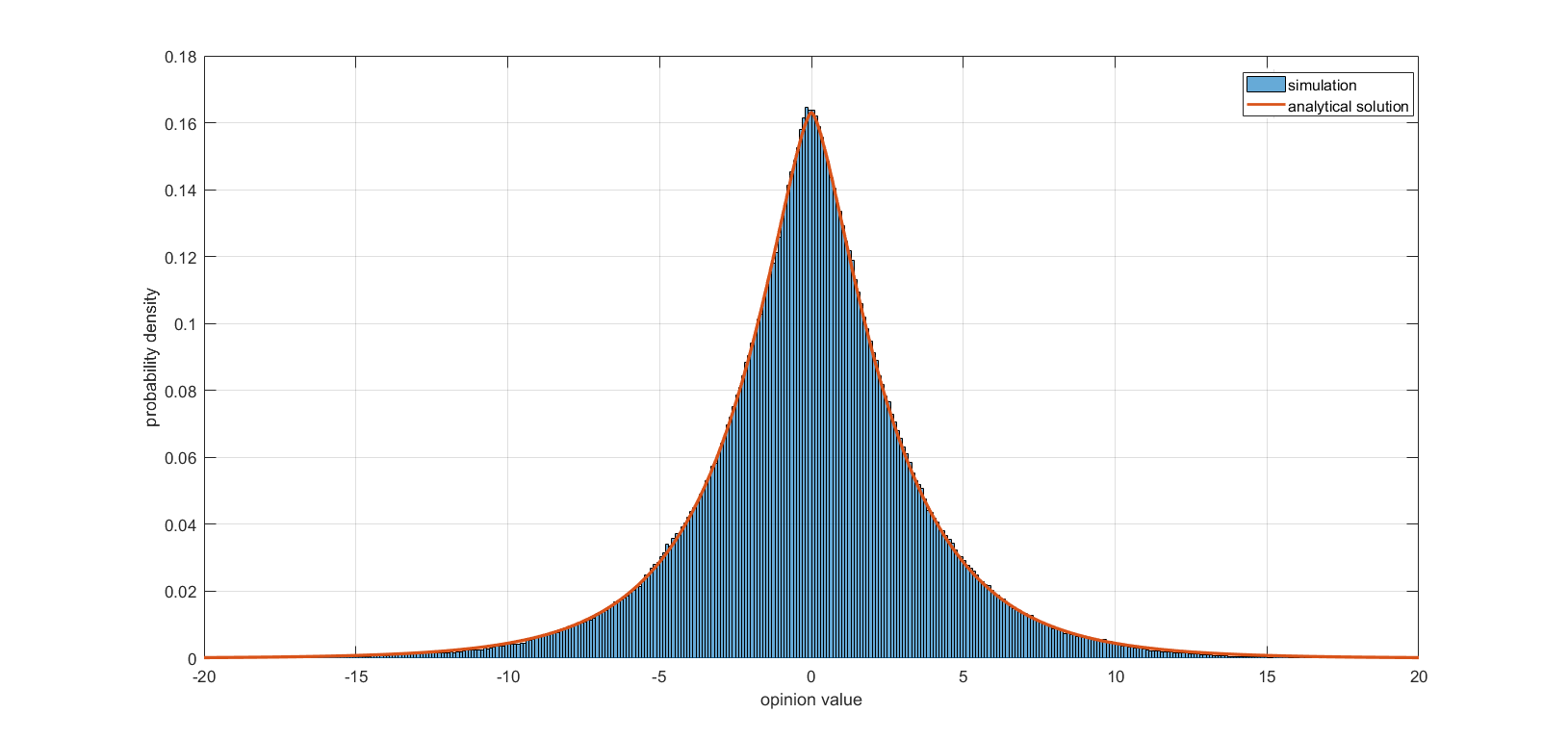} 
	\caption{Simulation vs explicit solution when $\lambda=2.0$, $\sigma=3.0$, and $\alpha=0.5$.}
	\label{fig:simulation2}
\end{figure}
As expected, $p(x)$ does not depend on the initial state $X(0)$. The function $p(x)$ is non-negative
and bounded when $0\leq \alpha<2$. Non-negativity can be easily seen as the $\text{BesselK}(\nu,z)$ 
function is non-negative. Boundedness follows from the fact that
for $\theta>1$, since $a_n\to 0$ when $n\to \infty$, 
\begin{align*}
    p(x) \leq \sum_{n=0}^{\infty} \frac{C}{\theta^{n/2}} < +\infty,
\end{align*}
where $C$ is a universal constant. 

Since $M(1) = \frac{1}{2}$ from \eqref{eq_mellin}, we also have $\int p(x)dx = 1$.
A final remark is that $|x|^\frac{1}{2}\text{BesselK}\left(\frac{1}{2-\alpha},|x|^\frac{2-\alpha}{2}\right)$
has exponential tail as $|x|\to \infty$. Figure \ref{fig:simulation2} shows an example of the simulated histogram and the solution derived from Equation \eqref{eq_ode}.

\subsection{Parameter Sensitivity}
In this section, we analyze the sensitivity of the analytic solution
proposed in Theorem \ref{thm:main2body_dependent} with respect to the parameters of the model.
Figure \ref{fig:param_sensitivity} shows a collection of plots of the density $p(x)$
when varying each given parameter and fixing the other parameters.
The baseline parameters are $\alpha=0.5, \lambda=2.0, \sigma=3.0, \theta=2.0$.
Here is the summary of the effect of $\alpha, \lambda, \sigma$, and $\theta$.
\begin{figure}[!hb]
\centering
	\includegraphics[scale=0.3]{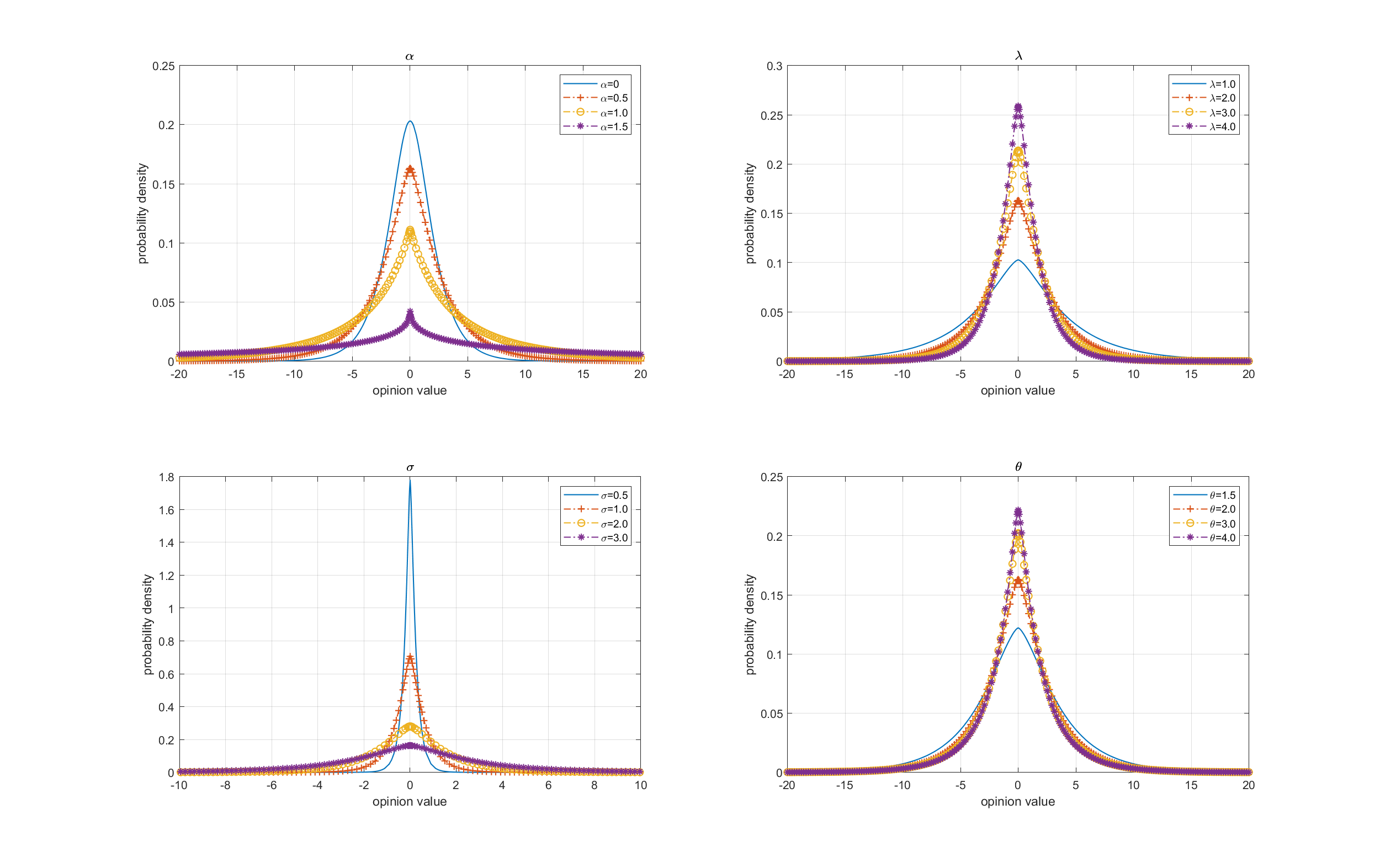} 
	\caption{Sensitivity of parameters $\alpha, \lambda, \sigma, \theta$.}
	\label{fig:param_sensitivity}
\end{figure}
\begin{itemize}
\item  (Left-top) $\alpha$ : When the regular opinion is sufficiently far away from the stubborn opinion,
interactions are less likely as already discussed. This is clearly reinforced when $\alpha$ is larger.
This intuitively explains why the tail of $p(x)$ is heavier as $\alpha$ increases.
We see that the peak point $p(0)$ is also decreasing as $\alpha$ increases.
\item (Right-top) $\lambda$ : As $\lambda$ increase, the interaction rate increases proportionally,
which implies more interactions, and a density $p(x)$ which is more concentrated around the
opinion of the stubborn agent.
\item (Left-bottom) $\sigma$ : As $\sigma$ increases, the strength of self-belief increases, 
which naturally results in a widened of the shape of $p(x)$.
\item (Right-bottom) $\theta$ : As $\theta$ increase, the weight of the opinion of the stubborn opinion
increases, which forces the regular opinion to be closer to that the stubborn opinion.
\end{itemize}

\subsection{More on Case {\bf {(C2)}} with a Bias Term}
We gather here partial results on the ordinary differential equation for $p_+(x)$ in
the presence of a non zero bias term. In this case,
\begin{align*}
	\sigma^2\frac{{\mathrm d}^2 p_+(x)}{{\mathrm d} x^2} -2 \mu \frac{{\mathrm d} p_+(x)}{{\mathrm d} x}  = 2\lambda p_+(x) -2\theta\lambda p_+(\theta x).
\end{align*}
The Mellin transform yields the recurrence equation.
\begin{align*}
	\sigma^2 s(s+1)M(s) +2\mu (s+1)M(s+1) =2\lambda \left(1-\frac{1}{\theta^{s+1}}\right) M(s+2).
\end{align*}
Let $f(s)$ be defined by
$$M(s)=f(s)\Gamma(s)\left(\frac{\sigma^2}{2\lambda}\right)^{s/2}.$$ 
Then $f(s)$ satisfies the relation
\begin{align*}
	f(s)+\frac{\sqrt{2}\mu}{\sigma\sqrt{\lambda}}f(s+1)=\left( 1-\frac{1}{\theta^{s+1}}\right) f(s+2).
\end{align*}
Let $\omega:=\frac{\sqrt{2}\mu}{\sigma\sqrt{\lambda}}$. We simplify the form.
\begin{align*}
f(s)+\omega f(s+1)=\left( 1-\frac{1}{\theta^{s+1}}\right) f(s+2).
\end{align*}
By introducing $\Psi(s):=\frac{f(s+1)}{f(s)}$, the equation can be re-written as
\begin{align}
\label{eqquad}
\frac 1 {\Psi(s)} + \omega = a(s+1)\Psi(s+1),
\end{align}
where $a(s):= 1- \theta^{-s}$. On the other hand, we can find a unique positive solution of the equation (by introducing $\gamma(s)$)
$$ \frac 1 {\gamma(s)} + \omega = a(s) \gamma(s),$$
it is easy to see that this function $\gamma(s)$ is non-decreasing.
Let $\Sigma(s)= \frac{\Psi(s)}{\gamma(s)} -1$. From \eqref{eqquad}, we may further simplify this as
\begin{align}
\label{eqquad2}
\Sigma(s) = \frac{\xi(s)}{\zeta(s)+\Sigma(s+1)},
\end{align}
where
\begin{eqnarray*}
\xi(s):= \frac{1}{a(s)\gamma(s)\gamma(s+1)}, \qquad \zeta(s):= 1-\frac{\gamma(s)}{\gamma(s+1)}\ge 0.
\end{eqnarray*}
Hence $\Sigma(s)$ admits the continued fraction expansion
\begin{eqnarray*}
\Sigma(s) = \frac {\xi(s)} 
          {\zeta(s) + \frac {\xi(s+1)}
                         {\zeta(s+1) +\frac{\xi(s+2)}
                                      {\zeta(s+2) + \frac{\xi(s+3)}{\zeta(s+3)+\cdots}
                                      }
                         }
          }.
\end{eqnarray*}
From the definition of $\Psi(s)$ above, it follows that $f(s+1)=f(s)\gamma(s)(1+\Sigma(s))$.
Therefore, we may expand $f(s)$ as follows
\begin{align*}
f(s)= \phi \prod_{k=0}^\infty \frac 1{\gamma(s+k)(1+\Sigma(s+k))},
\end{align*}
where $\phi$ is a constant. Then we have a representation of $M(s)$.

\subsection{Summary of the Results and Questions on the Two-Agent Model}
\label{sec3.6}
Here we the list results obtained so far, some properties of interest, and some open-problems.
\begin{itemize}
\item For {\bf (C1)} with any bias term $\mu$, the stochastic process $X(t)$ admits a unique stationary regime. This stationary regime is ergodic (as a factor of a marked Poisson point process). We give a probabilistic representation of the stationary distribution  
in Proposition \ref{pro:probilistic_solution}.
For $\mu=0$, we also give an analytical solution given in
Theorem \ref{thm:main2body_dependent} for $\alpha=0$, namely
\begin{align*}
	p\left(x \right)=\phi \sum_{n=0}^{\infty} \frac{a_n}{\theta^n} e^{-\left(\frac{\sqrt{2\lambda}2\theta^n}{\sigma}\right)|x|},
\end{align*}
where $a_n=\prod_{k=1}^{n} \left(\frac{\theta^2}{1-\theta^{2k}} \right)$ with $a_0=1$ and $\phi = \frac{\sqrt{2\lambda}}{2\sigma}\left(\prod_{k=0}^{\infty}\left( 1- \frac{1}{\theta^{2(1+k)}} \right)\right)^{-1}$. Moreover, 
Equation \eqref{eq_mellin} can be simplified as
\begin{align}
\label{eq_mellin_zero_alpha}
	M(s)=\phi \Gamma(s)\left( \frac{\sigma^2}{2\lambda}\right)^{s/2} \prod_{k=0}^{\infty} \left( 1- \frac{1}{\theta^{s+1+2k}} \right).
\end{align}
\item For {\bf (C2)} with $\mu=0$ and $0\leq \alpha < 1$,
there exists a unique stationary regime for $X(t)$ and
this stationary process is ergodic. This follows from the fact that the process
sampled at jump times is a $\phi$-irreducible Markov chain \cite{meyn2012markov}.  
The stationary distribution of this stationary regime is given in
Theorem \ref{thm:main2body_dependent}.
We also have an analytical solution to the ODE characterizing the stationary regimes when $1\leq \alpha< 2$.
See Theorem \ref{thm:main2body_dependent}. However, we cannot connect this solution to
a dynamics defined pathwise. An interesting open question is about the meaning of this analytic solution
for $\alpha$ in this range.
\item For {\bf (C3)}, $X(t)$ is pathwise well-defined since the
stochastic intensity is bounded. The Markov analysis is of the same nature
as that alluded to above. The associated process is ergodic when $\alpha <2$.
This leads to a natural (open) question. Let the solution of $p_L(x)$ in \eqref{eq_ode}
with $\lambda_L(x)$. Do we have $\lim_{L\to\infty} p_L(x) = p(x)$ in Theorem \ref{thm:main2body_dependent},
e.g., when $1\leq \alpha <2$?
\end{itemize}

\section{Extension to Multi-agent and Multi-dimensional Models}\label{multi}
This section focuses on some extensions of the models to 1)  multiple agents and 2) 
multi-dimensional opinions. For the first extension, one can come up with many interesting
scenarios, e.g., based on a social interaction graphs. We illustrate the flexibility of
our opinion independent approach by solving a specific scenario with three agents.  
We then discuss higher dimensional opinions, as well as a natural mean-field model
where the opinion independent and opinion dependent interaction rate models are 
connected through a single model.

\subsection{A Three-agent Interaction Model}
Consider a scenario with three agents $X_1$, $X_2$, and $X_3$. Agents $X_1$ and $X_3$ are
stubborn with opinion values $X_1(t)=s_1\in\mathbb{R}$ and $X_3(t)=s_3\in \mathbb R$, whereas
Agent $X_2$ is regular. See Figure \ref{fig:three}.
\begin{figure}[h]
\begin{center}
\begin{tikzpicture}[scale=0.45,->,>=stealth',shorten >=1pt,mindmap,
  level 1 concept/.append style={level distance=130,sibling angle=30},
  extra concept/.append style={color=blue!50,text=black}]

  \begin{scope}[mindmap, concept color=orange, text=white]
    \node [scale = 0.4, concept] at (-7.5,0) (agent1) {Stubborn Agent $X_{1}=s_1$};
  \end{scope}

  \begin{scope}[mindmap, concept color=blue]

    \node [scale = 0.4,concept, text=white] at (0,0) (agent2) {Non-stubborn Agent $X_2$} ;
  \end{scope}
  
  \begin{scope}[mindmap, concept color=orange, text=white]
      \node [scale = 0.4, concept] at (7.5,0) (agent3) {Stubborn Agent $X_{3}=s_3$};
    \end{scope}


  \begin{scope}[draw=black,fill=black, decorate,decoration=circle connection bar]
     \path (agent1) edge (agent2);
     \path (agent3) edge (agent2);
  \end{scope}
\end{tikzpicture}
\caption{Three agent interaction model}\label{fig:three}
\end{center}
\end{figure}

Without loss of generality, we assume that $s_1<s_3$. Agent $X_2$'s diffusion has bias $\mu$
and variance $\sigma^2$. Agent $X_2$ interacts with the stubborn agent $X_1$ with intensity
$\Lambda_1(x)$ and with the stubborn agent $X_3$ with intensity $\Lambda_3(x)$.
For example, let $\Lambda_1(x)=\frac{\lambda_1}{|x-s_1|^{\alpha_1}}$ and
$\Lambda_3(x)=\frac{\lambda_3}{|x-s_3|^{\alpha_3}}$. The interactions of $X_2$ 
take place at the epochs of a point process $N_2(t)$ which is the superposition 
of two point processes $N_{1}(t)$ and $N_{3}(t)$ with stochastic intensities
$\Lambda_1(X_2(t-))$ and $\Lambda_3(X_2(t-))$ for interactions with $X_1$ and $X_3$, respectively.
At each interaction, $X_2$ updates its opinion by averaging it with the appropriate stubborn opinion.
We assume $\theta=\theta'=2$.

If both $\Lambda_1(X_2(t-))$ and $\Lambda_3(X_2(t-))$ are almost surely locally integrable,
we obtain the following non-local partial differential equation result.

\begin{proposition}\label{prop:threeagent}
Assume that for $i=1$ and $i=3$,
\begin{align*}
\mathbf{P}\left[\int_{a}^{b} \Lambda_i\left(X_2(t-)\right) {\mathrm d}t < +\infty,\quad \text{for any\ $0\leq a<b<+\infty$} \right]=1.
\end{align*}
Then, under the above assumptions, the density $p(t,x)\in C^2$ of $X_2(t)$ satisfies
the non-local partial differential equation:
	\begin{align*}
		\frac{\partial p(t,x)}{\partial t} = \frac{\sigma^2}{2}\frac{\partial^2 p(t,x)}{\partial x^2} &- \mu \frac{\partial p(t,x)}{\partial x}  - \left[\Lambda_1(x) + \Lambda_3(x)\right] p(t,x)\\
		& +2\Lambda_1(2x) p_t(2x-s_1) + 2\Lambda_3(2x) p_t(2x-s_3).
	\end{align*}
\end{proposition}
\begin{proof}{[Sketch of the proof]}
The proof follows the same lines of thought as in Theorem \ref{thm:evolution}.
Interaction terms are added as they result from conditionally independent interactions.
\end{proof}
By letting $\frac{\partial p(t,x)}{\partial t} =0$, we have the ordinary differential
equation for the stationary distribution.
\begin{corollary}
Under the assumptions of Proposition \ref{prop:threeagent},
the stationary density $p(x)$ of agent $X_2$ 
satisfies the non-local ordinary differential equation
\begin{align}\label{eq_ode2}
\sigma^2\frac{{\mathrm d}^2 p(x)}{{\mathrm d} x^2} +2 \mu \frac{{\mathrm d} p(x)}{{\mathrm d} x} =
2\left[\Lambda_1(x) +\Lambda_3(x) \right] p(t,x) -4\Lambda_1(2x) p_t(2x-s_1) - 4\Lambda_3(2x) p_t(2x-s_3).
\end{align}
\end{corollary}

We have no solution in the general power law case.
However, when $\mu=0$, $\Lambda_1(x)=q\lambda$, and $\Lambda_3(x)=(1-q)\lambda$ with $q\in[0,1]$,
i.e., in the opinion independent interaction rate case,
by following the same approach as in Section \ref{sec3.5}, the following probabilistic
representation of the solution can be obtained:
\begin{align}
X_2({+\infty}) = \left[ \sum_{j=0}^{\infty} \frac{V({j})}{2^j} + (s_3-s_1) \sum_{j=0}^{\infty} \frac{U(j)}{2^j} \right] + s_1,
\end{align}
where $V({j})$ follows an independent $N(\mu\Delta t_j, \sigma^2\Delta t_j)$ Gaussian distribution 
and $\Delta {t_j}\sim \text{Exponential}(\lambda)$ and $U(n)\sim \text{Bernoulli}(1-q)$. Let
$A= \sum_{j=0}^{\infty} \frac{V({j})}{2^j}$ and $B=  (s_3-s_1) \sum_{j=0}^{\infty} \frac{U(n)}{2^j}$.
Hence the stationary solution of $X_2(t)$ admits the following representation with two independent components:
\begin{align*}
X_2(+\infty)=A+B + s_1.
\end{align*}
We discussed the distribution of $A$ in Section \ref{sec3.5}.
Bhati et. al. studied the distribution of $B$ for $q\in[0,1]$ \cite{bhati2011distribution}.
The general form of $B$ is non-trivial. When $q=0.5$, $B \sim U[0,s_3-s_1]$ by matching each
realization of $B$ with the binary representation of real values in $[0,1]$ and multiplying
it by $s_3-s_1$, where $U[0,s_3-s_1]$ denotes the uniform distribution on $[0,s_3-s_1]$. So
$B+s_1\sim U[s_1,s_3]$.

\begin{proposition}\label{prop:three_body}
Let $g(x)=\frac{1}{s_3-s_1}$ be the density function of a uniform distribution on $[s_1,s_3]$. When
$\Lambda_1(x)=q\lambda$, $\Lambda_3(x)=(1-q)\lambda$, $\mu=0$ and $q=0.5$, the solution of \eqref{eq_ode2} is
\begin{align*}
p(x) = p^*(x) \star g(x)
\end{align*}
where $\star$ denotes convolution and $p^*(x)$ is the solution obtained in Theorem \ref{thm:main2body_dependent} 
when $\alpha=0$.
\end{proposition}

Figure \ref{fig:simulation3body} compares a simulated solution and the solution 
established in Proposition \ref{prop:three_body}.
\begin{figure}[h]
\centering
	\includegraphics[scale=0.3]{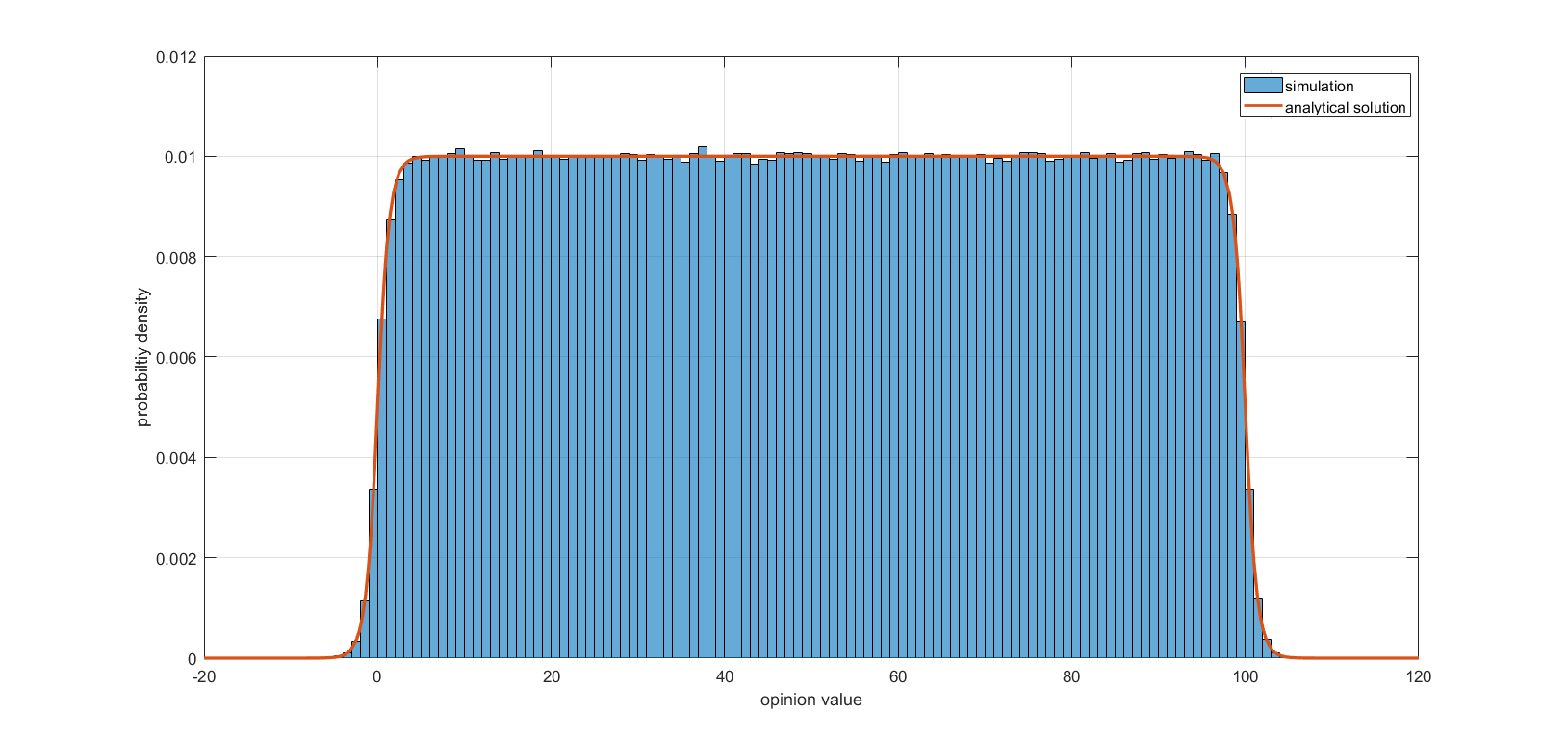}
	\caption{Three-body simulation vs analytically derived solution when $\lambda = 3.0$, $\sigma=2.0$, $q=0.5$, $\alpha=0$.}
	\label{fig:simulation3body}
\end{figure}

\subsection{Multi-dimensional Continuous Time Opinion Dynamics}
Let us come back to the two agent model. Assume that the regular agent $\mathbf{X}$ has
an opinion vector $\mathbf{X}(t) = \left(X_{1}(t), X_{2}(t),\cdots,X_{d}(t) \right)\in\mathbb{R}^d$ and updates its opinions by interacting with the stubborn agent. Assume that each component $X_{i}(t)$
follows an independent diffusion process with parameters $\mu_i$ and $\sigma_i$.
The stubborn agent $\mathbf{Z}$ has the opinion vector $\mathbf{Z}(t) =\left(0,0,\cdots, 0\right)\in\mathbb{R}^d$. 
The stochastic interactions are in terms of a multi-dimensional point process 
$\mathbf{N}(t)=(N_1(t), \cdots, N_d(t))$, with stochastic intensity $\Lambda_i(\mathbf{X}(t))$ for each $X_i$.
Assume that $\Theta'=(\theta_1',\cdots,\theta_d')$ and $1/\Theta'=(1/\theta_1',\cdots,1/\theta_d')$.
Let $\odot$ denote component-wise multiplication. Then
 \begin{align}\label{eq:dynamicsn}
 {\mathrm d}\mathbf{X}(t) = \boldsymbol{\mu} {\mathrm d}t + \boldsymbol{\sigma}\odot  {\mathrm d}\mathbf{W}(t) - \frac{1}{\Theta'} \odot \mathbf{X}({t-})  \odot \mathbf{N}({\mathrm d} t),
 \end{align}
where $\boldsymbol{\mu}=\left(\mu_1,\cdots,\mu_d \right)$, $ \boldsymbol{\sigma}=\left(\sigma_1,\cdots,\sigma_d\right)$. Note that
$\mathbf{X}({t-})$ is the vector-value approaching from the left of $\mathbf{X}(t)$.

As above, the only case that can be solved at this stage is the opinion-independent case. For instance,
when $\Lambda_i(\mathbf{X}(t))=\lambda_i$ and the components of $\mathbf{N}(t)$ are independent,
the components $X_i(t)$ are independent. Therefore, the distribution of $X_i(t)$ follows the
partial differential equation described in Theorem \ref{thm:evolution}, with $\mu_i$, $\sigma_i$,
$\theta_i$, and $\Lambda_i(x)=\lambda_i$. So the stationary distribution of $\mathbf{X}(t)$
is product form with marginals given by the
stationary distributions obtained in the opinion-independent case. 

When $\Lambda_i(\mathbf{X}(t))=\lambda$ and the components of $\mathbf{N}(t)$
are dependent Poisson point processes (for instance the very same point process
pathwise), then the {\emph{marginal}} stationary distributions of the components of
$\mathbf{X}(t)$ are available for each component, but we have no result on the
$d$-dimensional stationary distribution of $\mathbf{X}(t)$.


\subsection{A Mean-Field Interaction Model}\label{sec:mean_field}
In this subsection, we consider a $(d+1)$-agent model based on mean-field interactions.
This model features one stubborn agent with the zero-valued opinion (without loss of generality)
and $d$ regular agents. All regular agents are assumed to have the same dynamics in distribution.

Let $\mathbf{X}(t)=(X_1(t),\cdots, X_d(t))$ and $\mathbf{N}(t)=(N_1(t),\cdots, N_d(t))$. 
The dynamics of the regular agent $X_i(t)$ consists of an independent diffusion with bias $\mu$ and
diffusion coefficient $\sigma$. The main novelty is the assumption that $\{N_i(t)\}$ is a
collection of conditionally independent point processes with a {\em common}
stochastic intensity $\Lambda(x)$ of the form
\begin{align*} 
\Lambda(\mathbf{X}(t))= \frac{1}{d} \sum_{i=1}^d \frac {\widetilde \lambda}{|X_i(t)|^{\alpha}},
\end{align*}
where $\{N_i(t)\}$ are conditionally independent given $\mathbf{X}(t)$,
and where $\widetilde \lambda$ is a positive constant.
Hence the regular agents are only coupled by this shared stochastic intensity.

The model can be seen as a multi-dimensional variant of the model introduced in
the previous section where the interaction rate of agent $X_i$ with the stubborn agent is 
proportional to the {\em empirical moment} of order $-\alpha$ of the opinions of the regular agents.
For $i=1,\cdots,d$, each dynamics of $X_i$ is governed by the stochastic differential equation: 
\begin{align*} 
{\mathrm d}X_i(t) = \mu {\mathrm d}t+ \sigma {\mathrm d}W_i(t) - \frac{X_i(t-)}{\theta'}
{\mathrm d}N_i(t).
\end{align*}

Assume that when $d$ tends to infinity,
the steady state of $\Lambda(\mathbf{X}(t))$ tends to a positive and finite constant, say $\kappa$,
so that the regular agents become asymptotically independent when $d$ tends to infinity. 
This assumption will be referred to as the mean-field limit hypothesis below.

\begin{figure}[h]
\begin{center}
\begin{tikzpicture}[scale=0.5,->,>=stealth',shorten >=1pt,mindmap,
  level 1 concept/.append style={level distance=130,sibling angle=30},
  extra concept/.append style={color=blue!50,text=black}]

  \begin{scope}[mindmap, concept color=orange, text=white]
    \node [scale = 0.38, concept] at (0,0) (agent0) {Stubborn Agent $S$};
  \end{scope}

  \begin{scope}[mindmap, concept color=blue]

    \node [scale = 0.4,concept, text=white] at (0,5) (agent1) {Agent $X_1$} ;
    
    \node [scale = 0.4,concept, text=white] at (3.5,3.5) (agent2) {Agent $X_2$} ;

    \node [scale = 0.4,concept, text=white] at (5,0) (agent3) {Agent $X_3$};

    \node [scale = 0.4,concept, text=white] at (3.5,-3.5) (agent4) {Agent $X_4$};
    
    \node [scale = 0.4,concept, text=white] at (0,-5) (agent5) {Agent $X_5$};
    
    \node [scale = 0.4,concept, text=white] at (-3.5,-3.5) (agent6) {Agent $X_6$};

    \node [scale = 0.4,concept, text=white] at (-5,0) (agent7) {Agent $X_7$};
    
    \node [scale = 0.4,concept, text=white] at (-3.5,3.5) (agent8) {Agent $X_d$};
  \end{scope}


  \begin{scope}[draw=black,fill=black, decorate,decoration=circle connection bar]
     \path (agent0) edge node {} (agent1);
     \draw (agent0) edge node {} (agent2);
	 \path (agent0) edge node {} (agent3);
     \draw (agent0) edge node  {} (agent4);
  \draw (agent0) edge node {} (agent5);
    \draw (agent0) edge node {} (agent6);
      \draw (agent0) edge node {} (agent7);
        \draw (agent0) edge node {} (agent8);
  \end{scope}
\end{tikzpicture}
\caption{$(d+1)$-agent mean-field limit model} \label{fig:mean_field}
\end{center}

\end{figure}

For the following analytical derivation, we assume that $\mu=0$. Assuming that the mean-field limit exists, the constant $\kappa$ should coincide with the steady-state moment of order $-\alpha$ of the opinion of the regular agent
in the two-agent state-independent model analyzed in Section \ref{sec3.5}. 
When taking $\alpha=0$ and $\lambda=\kappa$ in the result of Theorem \ref{thm:main2body_dependent},
the Mellin tranform of the stationary $X_i$ becomes by \eqref{eq_mellin_zero_alpha}
\begin{align}
\label{eq_mellin22}
	M_{\kappa}(s)=\phi \Gamma(s)\left( \frac{\sigma^2}{2\kappa}\right)^{s/2} \prod_{k=0}^{\infty} \left( 1- \frac{1}{\theta^{s+1+2k}} \right),
\end{align}
where 
\begin{align*}
	\phi = \frac{\sqrt{2\kappa}}{2\sigma}\left(\prod_{k=0}^{\infty}\left( 1- \frac{1}{\theta^{2(1+k)}} \right)\right)^{-1}.
\end{align*}
When it exists, the mean-field limit should hence satisfy some consistency equation:
the moment of order $-\alpha$ of the density $p(x)$ given in Theorem \ref{thm:main2body_dependent}
for the parameters $\theta,\sigma$ and $\lambda=\kappa$ should be such that
\begin{align}\label{eq:consistency}
2\widetilde \lambda M_{\kappa}(1-\alpha)=\kappa.
\end{align}
In view of \eqref{eq_mellin22}, when $0 < \alpha <1$,
Equation \eqref{eq:consistency} can be rewritten as
$$c\cdot \kappa^{\frac{\alpha}{2}}=\kappa,$$
for some constant $0<c<+\infty$ (the fact that this constant is finite requires the
assumption that $\alpha<1$ because of the singularity of the $\Gamma$ function). Hence,
for all $0<\alpha < 1$, the only positive and finite solution of this consistency
equation is $\kappa=c^{2/(2-\alpha)}$. 

We conclude that this mean-field limit, when it holds has a unique and well defined
solution for all $0<\alpha < 1$. It is beyond the scope of the present paper
to prove that the mean-field limit holds. However, let us stress that there
is numerical evidence that the mean-field hypothesis holds for all $0<\alpha < 1$. 

In contrast, when $\alpha \ge 1$, there is no non-degenerate solution to the self-consistency
equation, and there is in addition numerical evidence that the mean-field hypothesis does not hold. 
That is, when $d$ tends to infinity, the empirical moment of order $-\alpha$ of the
regular agent opinions does not tend to a finite limit.

The statements on the case $0<\alpha < 1$ is illustrated in Figure \ref{fig:mean_field_vard} which plots
the plots the empirical histrogram of the opinions of the $d$ regular agents
for various choices of $d$ when $\alpha=1/2$. When $d\geq 10000$,
the simulated histogram is very close to the explicit solution.
Note that the mean-field limit is a very good approximation for
much smaller valued of $d$.

\begin{figure}[h]
\centering
	\includegraphics[scale=0.35]{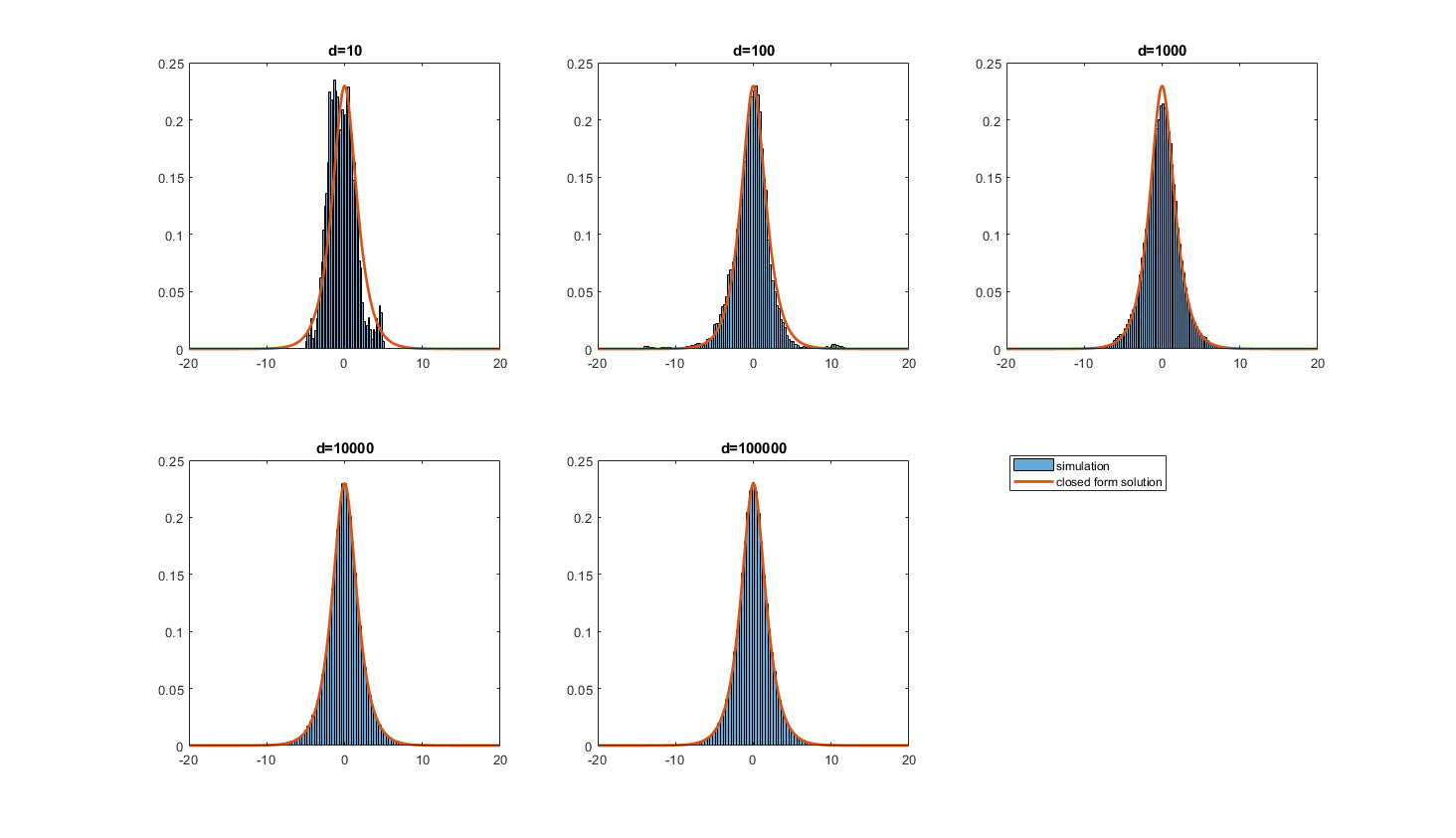} 
	\caption{Simulation histograms and analytically computed density of mean-field limit when $d=10,100, 1000, 10000, 100000$ and $\lambda' =2.0$, $\sigma=3.0$, $\alpha=0.5$, and $\theta=\theta'=2.0$.}
	\label{fig:mean_field_vard}
\end{figure}

\section{Conclusion}
In this paper, we introduced a new continuous time model for opinion
dynamics that features power law confidence between agents subject to
diffusive forces. 
The steady state behavior of this type of dynamics was shown to
satisfy to non-local partial differential equations 
and was characterized using Mellin transforms. 
We first solved the two agent problem and then proposed some extensions to a multi-agent cases,
including a mean-field model that captures the essence of the rate dependent model.
In the presence of diffusive self-beliefs, this model fundamentally differs from the bounded confidence model
in that it leads to weak consensus for small enough interact exponents.
It also qualitative differs from that
of discrete-time models due to the possibility of accumulation of interaction events.
Our analysis leads to a good understanding of the case where the interaction
exponent is less than 1. In this case, there is no such accumulation of interaction
events, we give a pathwise construction of the dynamics and obtained an explicit formula for the stationary
distribution in question. The case where this exponent is between 1 and 2
remains mysterious as there is no pathwise construction for the dynamics
and yet some distributional solution to the non-local partial differential equation.

\appendix
\section{Proof of smoothness of the density}\label{sec:appendix1}
\begin{lemma}\label{lem:sato}
Let $\mathcal{F}_X(\xi)$ denote the characteristic function of the random variable $X$ on $\mathbb{R}$. If the characteristic function of $X$ satisfies
\begin{align*}
    \int_{\mathbb{R}} |\xi|^m|\mathcal{F}_X(\xi)| \mathrm{d}\xi < +\infty,
\end{align*}
for some $m\in\mathbb{N}$, then the density $f(x)$ of $X$ is of class $C^m$ and the derivatives of orders $0,\cdots, m$ of $f(x)$ converge to $0$ as $|x|\to \infty$, where $f(x)\in C^m$ means the $m$-th order derivative $f^{(m)}(x)$ exists for all $x\in \mathbb{R}$.
\end{lemma}
\begin{proof}
See \cite[Proposition 28.1]{ken1999levy}.
\end{proof}
\begin{lemma}\label{lem:smooth_density}
Assume $\theta>1$. Under ${\bf H}$ (from Section \ref{sec:fpe}), $X(t)$ has a smooth density $p(t,x)$ which satisfies
\begin{align}
p(t,x)\in C^{1,\infty}\left(\mathbb{R}\right)\text{ and, } \forall m\geq 1, \frac{\partial^m p}{\partial x^m}(t,x)\to 0 \text{ as $|x| \to\infty$,}
\end{align}
where $p(t,x)\in C^{1,\infty}$ means $p(t,x)$ is differentiable with respect to $t$, and $p(t,x)$ is differentiable with respect to $x$ infinitely many times.
\end{lemma}
\begin{proof}
Under ${\bf H}$, the point process $N$ exists and does not have accumulation points as proved in Theorem \ref{thm:separation}. Conditionally on $N_n(t)=\delta_{t_1}+\cdots+\delta_{t_n}$ with $0< t_1 < \cdots < t_n \leq t$,
\begin{align*}
X(t) = \sigma\left[ W(t)-W(t_n) + \frac{W(t_n)-W(t_{n-1})}{\theta} + \cdots + \frac{W(t_1) }{\theta^{n}}\right] + \frac{x_0}{\theta^{n}},
\end{align*}
so that the characteristic function of $X(t)$ is
\begin{align*}
    \mathcal{F}_n(\xi):=&\mathbf{E}\left[ e^{i\xi X(t)} \big| \{N_n(t)=\delta_{t_1}+\cdots+\delta_{t_n}\} \right]\\
    =& \exp{\left[ i\frac{x_0}{\theta^n}\xi -\frac{\sigma^2\xi^2}{2}\left(t-t_{n} + \frac{t_n-t_{n-1}}{\theta^2} + \cdots + \frac{t_{1}}{\theta^{2n}}\right) \right]}.
\end{align*}
Since $\theta>1$,
\begin{align*}
    |\mathcal{F}_n(\xi)| \leq \exp{\left(-\frac{\sigma^2\xi^2 t}{2\theta^{2n}}\right)}.
\end{align*}
Denote by $\mathcal{F}_{X(t)}(\xi)$ the characteristic function of $X(t)$ and by $\nu_n(\cdot)$ the Janossy measure of $N[0,t]$ (see \cite{daley2007introduction, baccelli2016entropy}). Then
\begin{align*}
    |\mathcal{F}_{X(t)}(\xi)| &\leq \sum_{n=0}^{\infty} \int_{\{0< t_1 < \cdots < t_n < t\} }|\mathcal{F}_n(\xi)| {\mathrm d}\nu_n ({\mathrm d}t_1\cdots {\mathrm d}t_n) \\
    &\leq \sum_{n=0}^{\infty} \exp{\left(-\frac{\sigma^2\xi^2 t}{2\theta^{2n}}\right)} \mathbf{P} \left[ N(t)=n \right]. 
\end{align*}
Hence for each $m\in \mathbb{N}$:
\begin{align*}
    \int_{\mathbb{R}} |\xi|^{m} |\mathcal{F}_{X(t)}(\xi)| {\mathrm d}\xi &\leq \int_{\mathbb{R}} |\xi|^{m}\sum_{n=0}^{\infty}\exp{\left(-\frac{\sigma^2\xi^2 t}{2\theta^{2n}}\right)} \mathbf{P} \left[ N(t)=n \right]{\mathrm d}\xi  \\
    &\leq \sum_{n=0}^{\infty} \frac{C(\sigma, m, t)}{\theta^{n(m-1)}} \mathbf{P} \left[ N(t)=n \right] \leq \sum_{n=0}^{\infty} \frac{C(\sigma, m, t)}{\theta^{n(m-1)}} < + \infty,
\end{align*}
for some constant $C(\sigma, m, t)$ \cite[Section 11.2]{krishnamoorthy2016handbook}. Applying Lemma \ref{lem:sato} concludes the $m$-th order differentiability and the decay on tails for each order of the partial derivative with respect to $x$ .

By \cite[Proposition 6.1.1]{bjork2011introduction}, $X(t)$ is a Markov process, then it satisfies Chapman-Kolmogorov equation. The differentiability with respect to $t$ is implied by applying Chapman-Kolmogorov equation. 
\end{proof}

\section*{Acknowledgements}
We thank the editor and anonymous reviewers for their constructive comments, which helped us to improve the manuscript. This research was funded by Department of Defense \#W911NF1510225 and by a Math+X award from the Simons Foundation \#197982 to The University of Texas at Austin. The work of Fran\c{c}ois Baccelli was supported by the ERC grant 788851.

%
%
%
%






\bibliographystyle{plain}
\bibliography{bibbib}

\end{document}